\documentclass[reqno,11pt,letterpaper]{amsart}
\usepackage{mathrsfs}
\usepackage{amssymb}
\usepackage[usenames,dvipsnames]{xcolor}
\usepackage{amsthm}
\usepackage{amsmath}
\usepackage{amsfonts}
\usepackage{hyperref}
\usepackage{enumerate}
\usepackage{txfonts}
\usepackage{bm}
\usepackage{graphicx}
\numberwithin{equation}{section}
\usepackage{tikz}
\usetikzlibrary{3d,perspective,arrows.meta,calc}
\usepackage{float}
\usepackage{cite}

\newtheorem{theorem}{Theorem}[section]
\newtheorem{lemma}[theorem]{Lemma}
\newtheorem{corollary}[theorem]{Corollary}
\newtheorem{proposition}[theorem]{Proposition}

\theoremstyle{definition}
\newtheorem{definition}[theorem]{Definition}
\newtheorem{assumption}[theorem]{Assumption}
\newtheorem{example}[theorem]{Example}

\theoremstyle{remark}
\newtheorem{remark}[theorem]{Remark}

\allowdisplaybreaks 

\begin{document}
	
	\title[transonic shock]{Three-dimensional steady transonic shocks for the relativistic Euler equations}
\author{Shangkun WENG} 
\address{School of Mathematics and Statistics, Wuhan University, Wuhan, 430072, Hubei Province, People's Republic of China.}
\email{skweng@whu.edu.cn}

\author{Yan ZHOU}
\address{School of Mathematical Sciences, Fudan University, Shanghai, 200433, People's Republic of China.} 
\email{yanzhou@fudan.edu.cn.}

\date{}
\maketitle

\def\be{\begin{eqnarray}}
	\def\ee{\end{eqnarray}}
\def\ba{\begin{aligned}}
	\def\ea{\end{aligned}}
\def\bay{\begin{array}}
	\def\eay{\end{array}}
\def\bca{\begin{cases}}
	\def\eca{\end{cases}}
\def\p{\partial}
\def\no{\nonumber}
\def\e{\epsilon}
\def\de{\delta}
\def\De{\Delta}
\def\om{\omega}
\def\Om{\Omega}
\def\f{\frac}
\def\th{\theta}
\def\la{\lambda}
\def\lab{\label}
\def\b{\bigg}
\def\d{\mathrm{d}}
\def\var{\varphi}
\def\na{\nabla}
\def\ka{\kappa}
\def\al{\alpha}
\def\La{\Lambda}
\def\ga{\gamma}
\def\Ga{\Gamma}
\def\ti{\tilde}
\def\wti{\widetilde}
\def\wh{\widehat}
\def\ol{\overline}
\def\ul{\underline}
\def\Th{\Theta}
\def\si{\sigma}
\def\Si{\Sigma}
\def\oo{\infty}
\def\q{\quad}
\def\z{\zeta}
\def\co{\coloneqq}
\def\eqq{\eqqcolon}
\def\di{\displaystyle}
\def\bt{\begin{theorem}}
	\def\et{\end{theorem}}
\def\bc{\begin{corollary}}
	\def\ec{\end{corollary}}
\def\bl{\begin{lemma}}
	\def\el{\end{lemma}}
\def\bp{\begin{proposition}}
	\def\ep{\end{proposition}}
\def\br{\begin{remark}}
	\def\er{\end{remark}}
\def\bd{\begin{definition}}
	\def\ed{\end{definition}}
\def\bpf{\begin{proof}}
	\def\epf{\end{proof}}
\def\bex{\begin{example}}
	\def\eex{\end{example}}
\def\bq{\begin{question}}
	\def\eq{\end{question}}
\def\bas{\begin{assumption}}
	\def\eas{\end{assumption}}
\def\ber{\begin{exercise}}
	\def\eer{\end{exercise}}
\def\mb{\mathbb}
\def\mbR{\mb{R}}
\def\mbZ{\mb{Z}}
\def\mc{\mathcal}
\def\mcS{\mc{S}}
\def\ms{\mathscr}
\def\lan{\langle}
\def\ran{\rangle}
\def\lb{\llbracket}
\def\rb{\rrbracket}
\begin{abstract}
	We investigates the transonic shock problem for steady relativistic Euler equations in hemispherical shells. We show the existence and uniqueness of spherically symmetric transonic shock solutions via the shooting method and the monotonicity between the exit pressure and shock position. Furthermore, without any restrictions on background transonic shocks, we prove the existence and stability of transonic shocks under three dimensional perturbations of the exit pressure. The key techniques involve the decoupling the hyperbolic and elliptic components in steady relativistic Euler equations using the deformation tensor and vorticity. To address the coordinate singularities, the ``spherical projection coordinates" combining the spherical coordinates and the stereographic projection is employed. Subsequently, the Rankine-Hugoniot conditions are appropriately reformulated to determine the shock front and derive the boundary conditions on the shock front for a first-order nonlocal deformation-curl system.
\end{abstract}


\section{Introduction and main results}

For the classical convergent-divergent De Laval nozzle, Courant and Friedrichs \cite{CF1948} described that when the incoming flow accelerates to supersonic speed after passing through the throat, a shock front appears in the diverging section to match the appropriately prescribed exit pressure, and the flow is compressed and decelerated to subsonic speed across the shock. In this paper, we study the spherical transonic shock problem for steady relativistic Euler equations in a hemispherical shell $\Omega=\{(x_1,x_2,x_3)\in \mathbb{R}^3 \mid x_1>0,\  r_1<r=\sqrt{x_1^2+x_2^2+x_3^2}<r_2\}$ (see Figure~\ref{fig:shell}), whose governing equations are as follows
\begin{equation}\label{rela-euler-1}
	\begin{cases}
		\sum_{i=1}^{3} \partial_{x_i}\!\left(\dfrac{n U_i c}{\sqrt{c^2 - U^2}}\right) = 0, \\
		\sum_{i=1}^{3} \partial_{x_i}\!\left(\dfrac{(p + \rho c^2) U_i U_j}{c^2 - U^2} + p\,\delta_{ij}\right) = 0, \quad j = 1,2,3, \\
		\sum_{i=1}^{3} \partial_{x_i}\!\left(\dfrac{(p + \rho c^2) U_i}{c^2 - U^2}\right) = 0,
	\end{cases}
\end{equation}
where $n$ is the rest mass density, ${\bf U} = (U_1, U_2, U_3)^{\!\top}$ is the velocity with $U^2 = U_1^2 + U_2^2 + U_3^2$, the constant $c$ is the speed of light, $p$ is the pressure, and $\rho$ is the mass-energy density satisfying 
\begin{equation}\label{rho}
	\rho = n \left( 1 + \frac{e}{c^2} \right),
\end{equation}
with $e$ being the specific internal energy. For an ideal polytropic gas, the specific internal energy satisfies
\begin{equation}\label{eq:internal-energy}
	e = C_v T= \frac{p}{(\gamma - 1)n},
\end{equation} 
 where $C_v = \frac{R}{\gamma - 1}$ is the specific heat at constant volume,  $R > 0$ is the gas constant, $\gamma > 1$ is the adiabatic exponent, and  $T$ is the temperature.  The equation of state is of the form
\begin{equation}\label{p}
	p = K(S) n^{\gamma}.
\end{equation}


Furthermore, the relativistic Bernoulli quantity is defined as
\begin{equation}\label{rel-bernoulli}
	\mathcal{B} = \frac{p+\rho c^2}{n \sqrt{c^2 - U^2}}.
\end{equation}

Since rewriting \eqref{rela-euler-1} in spherical coordinates may give rise to ``singular terms” that pose several difficulties, we introduce the ``spherical projection coordinates" $(y_1,y_2,y_3)$ as below
\begin{equation}\label{eq:spherical_transformation}\begin{aligned}
		y_1 & = \sqrt{x_1^2 + x_2^2 + x_3^2}, \\
		y_2 &= \frac{x_2}{x_1 + \sqrt{x_1^2 + x_2^2 + x_3^2}} , \\
		y_3 &= \frac{x_3}{x_1 + \sqrt{x_1^2 + x_2^2 + x_3^2}} .
	\end{aligned}
\end{equation}
The inverse transformation is given by
\begin{equation*}
	x_1 = \frac{1 - |y'|^2}{1 + |y'|^2}\, y_1,\quad
	x_2 = \frac{2 y_1 y_2}{1 + |y'|^2},\quad
	x_3 = \frac{2 y_1 y_3}{1 + |y'|^2},
\end{equation*}
where $y'=(y_2,y_3)$. Under this transformation, the hemispherical shell $\Omega$ is mapped onto the cylindrical domain  
\begin{equation*}
	\mathcal Y \coloneqq \big\{ (y_1,y') \mid r_1<y_1<r_2,\ y' \in E \big\},
\end{equation*}
where $E \coloneqq \{y' \in\mathbb{R}^2 \mid |y'| < 1 \}$ (see Figure~\ref{fig:cylinder}).

\begin{figure}[H]
	\centering
	\noindent
	\begin{minipage}[b]{0.49\textwidth} 
		\centering
		\begin{tikzpicture}[scale=0.65, >=Stealth] 
			\def\R{4}      
			\def\r{2.2}    
			\def\Rz{1.2}   
			\def\rz{0.66}  
			
			\draw[dashed, draw=gray!80, thick] (\R, 0) arc (0:180:{\R} and {\Rz});
			\draw[dashed, draw=gray!80, thick] (\r, 0) arc (0:180:{\r} and {\rz});
			
			\fill[gray!20, opacity=0.6, even odd rule] 
			(0,0) ellipse ({\R} and {\Rz})
			(0,0) ellipse ({\r} and {\rz});
			
			\fill[gray!30, opacity=0.7]
			(\R,0) arc (0:180:\R) -- (-\r,0) arc (180:0:\r) -- cycle;
			
			\draw[thick, black] (\R, 0) arc (0:180:\R);
			\draw[thick, black] (\r, 0) arc (0:180:\r);
			
			\draw[thick, black] (\R, 0) arc (0:-180:{\R} and {\Rz});
			\draw[thick, black] (\r, 0) arc (0:-180:{\r} and {\rz});
			
			\draw[->, thick] (0,0) -- (0, \R+1) node[above] {$x_1$};
			\draw[->, thick] (0,0) -- (\R+1, 0) node[right] {$x_2$};
			\draw[->, thick] (0,0) -- (-2, -1.5) node[below left] {$x_3$};
			
			\draw[->, dashed] (0,0) -- (45:\R) node[midway, right] {$r_2$};
			\draw[->, dashed] (0,0) -- (135:\r) node[midway, left] {$r_1$};
			
			\node[black, font=\large] at (60:3.2) {$\Omega$};
			\node[gray!80!black, font=\small] at (3.5, -0.6) {$x_1 = 0$};
			
			\useasboundingbox (-4.5, -2.0) rectangle (5.5, 5.5);
		\end{tikzpicture}
		\caption{\mbox{The hemispherical shell $\Omega$}}
		\label{fig:shell}
	\end{minipage}%
	\hfill
	\begin{minipage}[b]{0.49\textwidth} 
		\centering
		\begin{tikzpicture}[scale=0.65, >=Stealth] 
			\def\rOne{1.5}  
			\def\rTwo{5.5}  
			\def\R{2}       
			\def\Rz{0.6}    
			
			\draw[->, thick] (0,0) -- (\rTwo + 1.8, 0) node[right] {$y_1$};
			\draw[->, thick] (0,0) -- (0, \R + 1) node[above] {$y_2$};
			\draw[->, thick] (0,0) -- (-1.5, -1.2) node[below left] {$y_3$};
			
			\draw[dashed, draw=gray!80, thick] (\rOne, -\R) arc (-90:90:{\Rz} and {\R});
			
			\fill[gray!30, opacity=0.7] 
			(\rOne, \R) -- (\rTwo, \R) 
			arc (90:-90:{\Rz} and {\R}) -- (\rOne, -\R) 
			arc (270:90:{\Rz} and {\R}) -- cycle;
			
			\fill[gray!45, opacity=0.8] (\rTwo, 0) ellipse ({\Rz} and {\R});
			
			\draw[thick, black] (\rOne, \R) -- (\rTwo, \R);
			\draw[thick, black] (\rOne, -\R) -- (\rTwo, -\R);
			
			\draw[thick, black] (\rOne, \R) arc (90:270:{\Rz} and {\R});
			\draw[thick, black] (\rTwo, 0) ellipse ({\Rz} and {\R});
			
			\fill (\rOne, 0) circle (1.5pt) node[below] {$r_1$};
			\fill (\rTwo, 0) circle (1.5pt) node[below] {$r_2$};
			
			\node[black, font=\Large] at ({(\rOne+\rTwo)/2}, 0) {$\mathcal{Y}$};
			
			\draw[<->, dashed, thick, black] (\rTwo, 0) -- (\rTwo, \R) node[above] {$|y'|=1$};
			
			\node[black, font=\small] at (\rTwo - 0.2, -\R/2) {$E$};
			
			\useasboundingbox (-2.5, -2.8) rectangle (7.5, 4.7);
		\end{tikzpicture}
		\caption{The region $\mathcal{Y}$}
		\label{fig:cylinder}
	\end{minipage}
\end{figure}

Define the vector field ${\bf u} = (u_1, u_2, u_3)$ as
\begin{align*} 
	u_1 &= \frac{1 - |y'|^2}{1 + |y'|^2} U_1 + \frac{2 y_2}{1 + |y'|^2} U_2 + \frac{2 y_3}{1 + |y'|^2} U_3, \\
	u_2 &= - \frac{2 y_2}{1 + |y'|^2} U_1 + \frac{1 + y_3^2 - y_2^2}{1 + |y'|^2} U_2 - \frac{2 y_2 y_3}{1 + |y'|^2} U_3, \\
	u_3 &= - \frac{2 y_3}{1 + |y'|^2} U_1 - \frac{2 y_2 y_3}{1 + |y'|^2} U_2 + \frac{1 + y_2^2 - y_3^2}{1 + |y'|^2} U_3, \end{align*}
then  ${\bf U}(x)= u_1(y) \widehat{\bf e}_1(y) +u_2(y) \widehat{\bf e}_2(y) +u_3(y) \widehat{\bf e}_3(y)$,
with
\begin{equation*}
	\begin{aligned}
		&\widehat{\mathbf{e}}_1(y)  = \frac{1}{1 + |y'|^2}
		\begin{pmatrix}
			1 - |y'|^2  & 2y_2  & 2y_3 
		\end{pmatrix}^{\!\top}
		,\\
		& \widehat{\mathbf{e}}_2(y)  = \frac{1}{1 + |y'|^2}
		\begin{pmatrix}
			-2y_2 &
			1 + y_3^2 - y_2^2 &
			-2y_2 y_3
		\end{pmatrix}^{\!\top},\\
		&\widehat{\mathbf{e}}_3(y)  = \frac{1}{1 + |y'|^2}
		\begin{pmatrix}
			-2y_3 &
			-2y_2 y_3 &
			1 + y_2^2 - y_3^2
		\end{pmatrix}^{\!\top}.
	\end{aligned}
\end{equation*}
By construction, $\{\widehat{\mathbf e}_i\}_{i=1}^3$ is an orthonormal basis of $\mathbb{R}^3$. In the new coordinates, the Laplace operator, directional derivative, divergence, and curl take the following forms. Define $\psi(\mathscr{S}^{-1}y)=\widehat{\psi}(y)$ and ${\bf U}(\mathscr{S}^{-1} y)=u_1(y) \widehat{{\bf e}}_1(y) + u_2(y) \widehat{{\bf e}}_2(y) + u_3(y) \widehat{{\bf e}}_3(y)$, then
\begin{align}\label{y1}
	&\Delta_x \psi(x)= \partial_{y_1}^2\widehat{\psi}+\frac{(1+|y'|^2)^2}{4y_1^2}\sum_{j=2}^3\partial_{y_j}^2\widehat{\psi}+\frac{2}{y_1}\partial_{y_1}\widehat{\psi} ,\\\label{y2}
	&{\bf U}(x)\cdot\nabla_x= u_1\partial_{y_1}+\frac{1+|y'|^2}{2y_1}(u_2\partial_{y_2}+u_3\partial_{y_3}),\\\label{y3}
	&\text{div }{\bf U}(x)= \partial_{y_1}u_1(y) + \frac{1+|y'|^2}{2y_1}\sum_{j=2}^3\partial_{y_j}u_j+\frac{2u_1}{y_1} -\frac{1}{y_1}\sum_{j=2}^3 y_j u_j,\\\label{y4}
	&\text{curl }{\bf U}(x)=\omega_1(y) \widehat{{\bf e}}_1 +\omega_2(y) \widehat{{\bf e}}_2 +\omega_3(y) \widehat{{\bf e}}_3,
\end{align}
where
\begin{align*}
	&\omega_1= \frac{1+|y'|^2}{2y_1}(\partial_{y_2} u_3- \partial_{y_3} u_2) + \frac{1}{y_1}(y_3 u_2-y_2 u_3),\\
	&\omega_2=\frac{1+|y'|^2}{2y_1}\partial_{y_3} u_1- \partial_{y_1} u_3- \frac{u_3}{y_1},\\
	&\omega_3= \partial_{y_1} u_2+\frac{u_2}{y_1}- \frac{1+|y'|^2}{2y_1}\partial_{y_2} u_1.
\end{align*}

Under the ``spherical projection coordinates", the equations \eqref{rela-euler-1} becomes 
\begin{equation}\label{relativistic-euler-y}
	\begin{cases}
		\partial_{y_1}\!\left( \frac{n u_1}{\sqrt{c^2 - u^2}} \right)
		+ \frac{1 + |y'|^2}{2y_1} \left[ \partial_{y_2}\!\left( \frac{n u_2}{\sqrt{c^2 - u^2}} \right) 
		+ \partial_{y_3}\!\left( \frac{n u_3}{\sqrt{c^2 - u^2}} \right) \right]
		+ \frac{2}{y_1} \frac{n u_1}{\sqrt{c^2 - u^2}}
		- \frac{1}{y_1} \frac{n(y_2 u_2 + y_3 u_3)}{\sqrt{c^2 - u^2}} = 0, \\
		\left( u_1 \partial_{y_1} 
		+ \frac{1 + |y'|^2}{2y_1} \left( u_2 \partial_{y_2} + u_3 \partial_{y_3} \right) \right) u_1 
		- \frac{u_2^2 + u_3^2}{y_1}
		+ \frac{c^2 - u^2}{p + \rho c^2} \partial_{y_1} p = 0, \\[12pt]
		\left( u_1 \partial_{y_1} 
		+ \frac{1 + |y'|^2}{2y_1} \left( u_2 \partial_{y_2} + u_3 \partial_{y_3} \right) \right) u_2
		+ \frac{u_1 u_2}{y_1} 
		- \frac{u_3}{y_1} (y_3 u_2 - y_2 u_3)
		+ \frac{c^2 - u^2}{p + \rho c^2} \frac{1 + |y'|^2}{2y_1} \partial_{y_2} p = 0, \\
		\left( u_1 \partial_{y_1} 
		+ \frac{1 + |y'|^2}{2y_1} \left( u_2 \partial_{y_2} + u_3 \partial_{y_3} \right) \right) u_3
		+ \frac{u_1 u_3}{y_1} 
		+ \frac{u_2}{y_1} (y_3 u_2 - y_2 u_3)
		+ \frac{c^2 - u^2}{p + \rho c^2} \frac{1 + |y'|^2}{2y_1} \partial_{y_3} p = 0, \\
		\left( u_1 \partial_{y_1} 
		+ \frac{1 + |y'|^2}{2y_1} \left( u_2 \partial_{y_2} + u_3 \partial_{y_3} \right) \right) \mathcal{B}  = 0.
	\end{cases}
\end{equation}

\begin{remark}
	The set of ``spherical projection coordinates" \eqref{eq:spherical_transformation} is a composition of spherical coordinates and the stereographic projection, and was first introduced by Weng \cite{Weng25}. Under this coordinate transformation, the equations \eqref{relativistic-euler-y} contain no singular terms, which is beneficial for our subsequent stability analysis. Furthermore, we observe that the transformed equations are structurally similar to the steady Euler equations.
\end{remark}

First, we consider a class of spherically symmetric shock solutions to \eqref{relativistic-euler-y} with only a nontrivial radial velocity, for which the system reduces to  
\begin{equation}\label{rel-1d-euler}
	\begin{cases}
		\left( y_1^2 \dfrac{\bar{n} \bar{u}}{\sqrt{c^2 - \bar{u}^2}} \right)' = 0, \\
		\bar{u} \bar{u}' + \dfrac{c^2 - \bar{u}^2}{\bar{p} + \bar{\rho} c^2} \bar{p}' = 0, \\
		\bar{\mathcal{B}}' = 0.
	\end{cases}
\end{equation}
The corresponding Rankine-Hugoniot conditions and physical entropy condition at the shock front $y_1 = r_s$ are given by
\begin{equation}\label{RH-jump}
	\left[ \frac{\bar{n} \bar{u}}{\sqrt{c^2 - \bar{u}^2}} \right](r_s)
	= \left[ \frac{\bar{p} + \bar{\rho} c^2}{c^2 - \bar{u}^2} \, \bar{u}^2 + \bar{p} \right] (r_s)
	= [\bar{\mathcal{B}}] (r_s) = 0, \qquad
	\bar{S}^+ > \bar{S}^-.
\end{equation}
Here, $[f](r_s) \coloneqq f(r_s+) - f(r_s-)$ denotes the jump of $f$ across the shock front $y_1 = r_s$.

One can establish the existence and uniqueness of the transonic shock solution to the system \eqref{rel-1d-euler} for $y_1 \in [r_1, r_2]$, as stated in the following proposition.
\begin{proposition}\label{1d-existence}
	Assume that the supersonic flow is prescribed at the inlet $y_1 = r_1$, that is 
	\begin{equation}\label{inlet-cond}
		\bar{\mathbf{u}}(r_1) = (\bar{u}_0^- , 0, 0 )^{\!\top}, \quad \bar{n}(r_1) = \bar{n}_0^- > 0, \quad \bar{S}(r_1) = \bar{S}_0^- > 0,
	\end{equation}
	where $(\bar{u}_0^-)^2 > c_s^2(\bar{n}_0^-, \bar{S}_0^-)  > 0$ with $c_s = \sqrt{\partial_{\rho} p}$ being the local sound speed and $\bar{S}_0^-$ being a constant.
	Then, there exist two positive constants $P_1$ and $P_2$ ($P_1 < P_2$), depending only on the incoming supersonic flow and $r_1, r_2$, such that for any given exit pressure $p_e \in (P_1, P_2)$, the system \eqref{rel-1d-euler} admits a unique piecewise smooth shock solution
	\begin{equation}\label{solution-structure}
		\bar{\bm{\Psi}} (y_1) =
		\begin{cases}
			\bar{\bm{\Psi}}^- (y_1) \coloneqq (\bar{u}^-, 0, 0, \bar{n}^-, \bar{S}^-)(y_1), & y_1 \in [r_1, r_s), \\
			\bar{\bm{\Psi}}^+ (y_1) \coloneqq (\bar{u}^+,0, 0, \bar{n}^+, \bar{S}^+)(y_1), & y_1 \in (r_s, r_2],
		\end{cases}
	\end{equation}
	which satisfies the incoming supersonic state and the exit boundary condition $\bar{p}^+(r_2) = p_e$ with a shock front at $y_1 = r_s \in (r_1, r_2)$ satisfying \eqref{RH-jump}. Moreover, the shock front $y_1 = r_s $ increases as the exit pressure $p_e$ decreases.
\end{proposition}

\begin{remark}
	Since the proof is rather lengthy, we defer it to the Appendix.
\end{remark}


In what follows, the above solution  $ \bar{\bm{\Psi}}$ will be called the background solution. Clearly, the supersonic and subsonic parts of $ \bar{\bm{\Psi}}$ can be extended in a natural way, respectively. With an abuse of notations, we still call the extended subsonic and supersonic solutions $\bar{\bm{\Psi}}^+$ and $\bar{\bm{\Psi}}^-$, respectively.

Transonic shock problems in flat nozzles have been extensively investigated by many researchers, see \cite{CF03,Chen05,Chen08,CY08,FX21,FGXZ24,LS21, WZZ25, XY05,XY08} and references therein. For potential flows, the existence and stability of multidimensional transonic shocks with Dirichlet boundary condition was established by Chen and Feldman \cite{CF03}. Xin and Yin \cite{XY05} proved that the transonic shock problem using the potential flow model is ill-posed for the pressure condition at the exit. Chen \cite{Chen05} proved the stability of transonic shock fronts in two-dimensional nozzles under given pressure perturbations, which was generalized by Chen and Yuan \cite{CY08} to three-dimensional cylinders with square cross-sections. The well-posedness for non-isentropic potential flows in multidimensional divergent nozzles was established by Bae and Feldman \cite{BF11}. Fang and Xin \cite{FX21} investigated the existence and stability of two-dimensional steady compressible Euler flows in an almost flat finite nozzle. In particular, they found that the shock location can be uniquely determined by solving a free boundary problem for the linearized Euler system. More recently, Fang, Gao, Xiang and Zhao \cite{FGXZ24} established the existence of transonic shock solutions with large vorticity or large swirl velocity.


Now we discuss the state of the art for the stability of the transonic shocks in divergent sectors and conic nozzles. Li, Xin and Yin \cite{LXY09,LXY09MRL} established the existence and uniqueness of transonic shock solutions in 2D divergent nozzles with variable exit pressure. They \cite{LXY10} also proved the stability of transonic shocks in conical nozzles under the axi-symmetric perturbations of the exit pressure. For general two-dimensional de Laval nozzles, the authors \cite{LXY13} proved that transonic shock solutions to the full compressible Euler system exist and exhibit structural stability. Assuming some ``structural conditions" on spherically symmetric solutions, Liu, Xu and Yuan \cite{LXY16} proved the conditional stability of transonic shocks under multidimensional perturbations of exit pressure conditions.  Weng, Xie and Xin \cite{WXX21} introduced a reversible modified Lagrangian transformation and showed the stability of transonic shocks under axisymmetric perturbations of the nozzle shape. Recently, a significant progress was made by Weng and Xin \cite{WX26}, where they established the existence and stability of cylindrical transonic shocks under three-dimensional perturbations of the incoming flow and exit pressure. The structural stability of spherically symmetric transonic shock solutions in three-dimensional hemispherical shells was proved by Weng \cite{Weng25}, where he introduced the ``spherical projection coordinates", which we adopt as a crucial tool in this work. The deformation-curl decomposition developed by Weng and Xin \cite{WengXin19} played critical roles in \cite{WX26,Weng25}, and showed its robustness and universality in many circumstances, including the steady compressible MHD equations \cite{WY26} and smooth transonic flows \cite{WXCMP26}.

Let us also discuss some works for the unsteady and steady relativistic Euler equations. Pan and Smoller \cite{PS06} studied the finite-time blow-up of smooth solutions to the relativistic Euler equations. 
Ding and Li \cite{DL24, DL26} studied the global stability of rarefaction waves and large 1-shocks for two-dimensional steady supersonic flows passing Lipschitz wedges. 
For the 2D steady relativistic Euler equations, the authors in \cite{Fan23} characterized the properties of sonic--supersonic mixed flows and semi-hyperbolic patches. 
Furthermore, Lai \cite{Lai23b} extended the characteristic decompositions method to the axisymmetric relativistic Euler equations.

This paper will examine the stability of spherically symmetric transonic shock solutions in Proposition \ref{1d-existence} under small perturbations of exit pressure. The relativistic Rankine-Hugoniot conditions are
\begin{equation}\label{RH}
	\begin{cases}
		\left[ \dfrac{n u_1}{\sqrt{c^2 - u^2}} \right] 
		- \dfrac{1 + |y'|^2}{2\xi} \left( \partial_{y_2}\xi \left[ \dfrac{n u_2}{\sqrt{c^2 - u^2}} \right] 
		+ \partial_{y_3}\xi \left[ \dfrac{n u_3}{\sqrt{c^2 - u^2}} \right] \right) = 0, \\[16pt]
		\left[ \dfrac{(p + \rho c^2) u_1^2}{c^2 - u^2} + p \right] 
		- \dfrac{1 + |y'|^2}{2\xi} \left( \partial_{y_2}\xi \left[ \dfrac{(p + \rho c^2) u_1 u_2}{c^2 - u^2} \right] 
		+ \partial_{y_3}\xi \left[ \dfrac{(p + \rho c^2) u_1 u_3}{c^2 - u^2} \right] \right) = 0, \\[16pt]
		\left[ \dfrac{(p + \rho c^2) u_1 u_2}{c^2 - u^2} \right] 
		- \dfrac{1 + |y'|^2}{2\xi} \left( \partial_{y_2}\xi \left[ \dfrac{(p + \rho c^2) u_2^2}{c^2 - u^2} + p \right] 
		+ \partial_{y_3}\xi \left[ \dfrac{(p + \rho c^2) u_2 u_3}{c^2 - u^2} \right] \right) = 0, \\[16pt]
		\left[ \dfrac{(p + \rho c^2) u_1 u_3}{c^2 - u^2} \right] 
		- \dfrac{1 + |y'|^2}{2\xi} \left( \partial_{y_2}\xi \left[ \dfrac{(p + \rho c^2) u_2 u_3}{c^2 - u^2} \right] 
		+ \partial_{y_3}\xi \left[ \dfrac{(p + \rho c^2) u_3^2}{c^2 - u^2} + p \right] \right) = 0, \\[16pt]
		\left[ \dfrac{p + \rho c^2}{n \sqrt{c^2 - u^2}} \right] = 0.
	\end{cases}
\end{equation}
Moreover, the physical entropy condition is also satisfied
\begin{equation}\label{entropy-con}
	S^+ (\xi(y'), y') > S^- (\xi(y'), y') \quad \text{on } y_1 = \xi(y').
\end{equation}

The problem of selecting physically admissible shocks from Rankine-Hugoniot solutions was resolved for non-relativistic gases by Bethe \cite{Bethe} and Weyl \cite{Weyl}. The study of relativistic shock waves was pioneered by Taub \cite{Taub}, who introduced the celebrated Taub adiabat. Weyl's results were later extended to the relativistic case by Israel \cite{Israel}, and more recently, Asakura \cite{Asakura} investigated the R-H conditions to the relativistic Euler equations and developed the Bethe-Weyl theory for relativistic Euler systems systematically.

We prescribe the following boundary conditions. At the inlet $y_1 = r_1$, the incoming supersonic state is given by
\begin{equation}\label{inlet-per}
	\mathbf{\Psi}^- (r_1, y') = \bar{\mathbf{\Psi}}^- (r_1), \quad y' \in E.
\end{equation}
On the wall $\{(y_1, y')\mid r_1 \leq y_1 \leq r_2, |y'| = 1 \}$, the slip boundary condition yields
\begin{equation}\label{wall-bdry}
	y_2 u_2 (y_1, y') +  y_3 u_3 (y_1, y') = 0.
\end{equation}
At the exit $y_1 = r_2$, the pressure is prescribed as a small perturbation of the background pressure 
\begin{equation}\label{p-perturbation}
	p(r_2, y') = \bar{p}^+ (r_2) + \epsilon p_{ex} (y'), \quad y' \in E,
\end{equation}
where $\epsilon > 0$ is sufficiently small, $p_{ex} \in C^{2, \alpha} (\bar{E})$ satisfies the compatibility condition
\begin{equation}\label{com-p}
	( y_2 \partial_{y_2} + y_3 \partial_{y_3}) p_{ex} (y') = 0, \quad \text{on } |y'| = 1.
\end{equation}

The existence and uniqueness of transonic shock solutions satisfying the above boundary conditions can be proved.  Our main result is stated as follows.

\begin{theorem}\label{existence}
	{\it There exists $\epsilon_0>0$ depending only on the background solution $\overline{\bm{\Psi}}$ and the boundary data $p_{ex}\in C^{2,\alpha}(\overline{E})$ satisfying the compatibility conditions \eqref{com-p}, such that if $0\leq \epsilon<\epsilon_0$, the problem \eqref{relativistic-euler-y} with \eqref{RH}, \eqref{inlet-per}, \eqref{wall-bdry} and \eqref{p-perturbation} has a unique solution $\bm{\Psi}^+=(u_1^+,u_2^+,u_3^+,n^+,S^+)(y_1,y')$ with the shock front $\mathcal{S}: y_1=\xi(y')$ satisfying
		\begin{enumerate}[(i)]
			\item The function $\xi(y')\in C^{3,\alpha}(\overline{E})$ satisfies
			\begin{equation*}
				\|\xi(y')-r_s\|_{C^{3,\alpha}(\overline{E})}\leq C_*\epsilon,
			\end{equation*}
			where $C_*$ depends only on the background solution and the incoming flow and the exit pressure.
			\item The solution $\bm{\Psi}^+=(u_1^+,u_2^+,u_3^+,n^+,S^+)(y_1,y')\in C^{2,\alpha}(\overline{\mathcal{Y}_+})$ with $\mathcal{Y}_+=\{(y_1,y')\mid \xi(y')<y_1<r_2,y'\in E\}$ satisfies the entropy condition
			\begin{equation*}
				S^+(\xi(y')+,y') > S^-(\xi(y')-, y')\quad \text{for}\,\, y'\in E,
			\end{equation*}
			and
			\begin{equation*}
				\|(u_1^+-\bar{u}^+,u_2^+,u_3^+,n^+-\bar{n}^+,S^+-\bar{S}^+)\|_{C^{2,\alpha}(\overline{\mathcal{Y}_+})}\leq C_*\epsilon.
			\end{equation*}
		\end{enumerate}
}\end{theorem}


This paper is organized as follows. In Section~\ref{reformulation}, we reformulate the relativistic Euler equations. We decouple the elliptic and hyperbolic components of the equations by utilizing the deformation tensor and vorticity in Subsection~\ref{21}. The Rankine-Hugoniot conditions and boundary conditions are rewritten in Subsection~\ref{22}. Finally, we introduce a new coordinate transformation to transform the original problem into a fixed boundary value problem in Subsection~\ref{23}. Section~\ref{proof} is devoted to the proof of the main theorem. In the Appendix, we provide a proof of the existence and uniqueness of solutions to the one-dimensional problem.

\section{The reformulation of the relativistic Euler transonic shock problem}\label{reformulation}

\subsection{An effective decomposition of the steady relativistic Euler equations in terms of the deformation tensor and vorticity}\label{21}

We note that the steady relativistic Euler system is of hyperbolic-elliptic mixed type in subsonic regions, which brings difficulties in both solving the equations and establishing appropriate boundary conditions. To address this, we employ the deformation-curl decomposition introduced by the authors in \cite{WengXin19} to decouple the hyperbolic and elliptic quantities in the system. The details are as follows.

The Bernoulli's quantity and the entropy satisfy the transport equations
\begin{align}\label{rel-deform-1}
	&\left( \partial_{y_1} + \frac{1 + |y'|^2}{2y_1 u_1} (u_2 \partial_{y_2} + u_3 \partial_{y_3}) \right) K(S) = 0,\\ \label{rel-deform-2}
	& \left( \partial_{y_1} + \frac{1 + |y'|^2}{2y_1 u_1} (u_2 \partial_{y_2} + u_3 \partial_{y_3}) \right) \mathcal{B} = 0.
\end{align}

It follows from \eqref{y4}, the third and forth equations in \eqref{relativistic-euler-y} can be rewritten as
\begin{align*}
		& u_1 \omega_3 - u_3 \omega_1 + \frac{1 + |y'|^2}{2y_1} 
		\left(\frac{c^2 - u^2}{\mathcal{B}} \partial_{y_2} \mathcal{B} - \frac{\sqrt{c^2 - u^2}}{\gamma K \mathcal{B}}\left(\mathcal{B} \sqrt{c^2 - u^2} - c^2 \right) \partial_{y_2} K \right) = 0, \\
		& -u_1 \omega_2 + u_2 \omega_1 + \frac{1 + |y'|^2}{2y_1} 
		\left( \frac{c^2 - u^2}{\mathcal{B}} \partial_{y_3} \mathcal{B} - \frac{\sqrt{c^2 - u^2}}{\gamma K \mathcal{B}}\left(\mathcal{B} \sqrt{c^2 - u^2} - c^2 \right)  \partial_{y_3} K \right) = 0.
\end{align*}
Direct computation gives
\begin{equation}\label{rel-deform-4}
	\begin{aligned}
		\omega_2 &= \frac{u_2}{u_1} \omega_1 
		+ \frac{1 + |y'|^2}{2y_1 u_1} \left( \frac{c^2 - u^2}{\mathcal{B} } \partial_{y_3} \mathcal{B}  - \frac{\sqrt{c^2 - u^2}}{\gamma K \mathcal{B} }\left(\mathcal{B}  \sqrt{c^2 - u^2} - c^2 \right)  \partial_{y_3} K \right), \\
		\omega_3 &= \frac{u_3}{u_1} \omega_1 
		- \frac{1 + |y'|^2}{2y_1 u_1} \left(\frac{c^2 - u^2}{\mathcal{B} } \partial_{y_2} \mathcal{B}  - \frac{\sqrt{c^2 - u^2}}{\gamma K \mathcal{B} }\left(\mathcal{B}  \sqrt{c^2 - u^2} - c^2 \right) \partial_{y_2} K \right).
	\end{aligned}
\end{equation}

Since $\text{div } \text{curl }{\bf u} = 0$, then 
\begin{equation}\label{rel-div-omega}
	\partial_{y_1} \omega_1 + \frac{1 + |y'|^2}{2y_1} (\partial_{y_2} \omega_2 + \partial_{y_3} \omega_3) + \frac{2}{y_1} \omega_1 - \frac{1}{y_1} (y_2 \omega_2 + y_3 \omega_3) = 0.
\end{equation}
Combining with \eqref{rel-deform-4} yields
\begin{equation}\label{rel-deform-omega1}
	\begin{aligned}
		&\partial_{y_1} \omega_1 + \frac{1 + |y'|^2}{2y_1 u_1} \sum_{i=2}^{3} u_i \partial_{y_i} \omega_1 
		+ \left[ \frac{1 + |y'|^2}{2y_1} \sum_{i=2}^{3} \partial_{y_i} \left( \frac{u_i}{u_1} \right) + \frac{2}{y_1} - \frac{y_2 u_2 + y_3 u_3}{y_1 u_1} \right] \omega_1 \\
		&\quad + \frac{(1 + |y'|^2)^2}{4y_1^2} \Bigg[ 
		\partial_{y_2} \left( \frac{c^2 - u^2}{u_1 \mathcal{B} } \right) \partial_{y_3} \mathcal{B}  - \partial_{y_3} \left( \frac{c^2 - u^2}{u_1 \mathcal{B} } \right) \partial_{y_2} \mathcal{B}  \\
		&\qquad - \partial_{y_2} \left( \frac{\sqrt{c^2 - u^2}}{\gamma K u_1 \mathcal{B} }(\mathcal{B}  \sqrt{c^2 - u^2} - c^2) \right) \partial_{y_3} K \\
		&\qquad + \partial_{y_3} \left( \frac{\sqrt{c^2 - u^2}}{\gamma K u_1 \mathcal{B} }(\mathcal{B}  \sqrt{c^2 - u^2} - c^2) \right) \partial_{y_2} K \Bigg] = 0.
	\end{aligned}
\end{equation}

One can represent the density as a function of $\mathcal{B}$, $u^2$ and $S$ as
\begin{equation}\label{rel-deform-5}
	n = n (\mathcal{B}, u^2, S) =\left( \frac{\gamma - 1}{\gamma K(S)} \right)^{\frac{1}{\gamma - 1}} 
	\left( \mathcal{B} \sqrt{c^2 - u^2} - c^2 \right)^{\frac{1}{\gamma - 1}} .
\end{equation}
Substituting \eqref{rel-deform-5} into the continuity equation, we derive that
\begin{equation}\label{rel-deform-6}
	\begin{aligned}
		&\left( c^2 (c_s^2 - u_1^2) - c_s^2(u_2^2 + u_3^2)  \right) \partial_{y_1} u_1 
		+ \frac{1 + |y'|^2}{2y_1} \left(c^2 (c_s^2 - u_2^2) - c_s^2 ( u_1^2 + u_3^2) \right) \partial_{y_2} u_2 \\
		&+ \frac{1 + |y'|^2}{2y_1} \left(c^2 (c_s^2 - u_3^2) - c_s^2 (u_1^2 + u_2^2) \right) \partial_{y_3} u_3 
		+ \frac{c_s^2 (c^2 - u^2)}{y_1} \left( 2 u_1 - y_2 u_2 - y_3 u_3 \right) \\
		&=\ (c^2 - c_s^2) \Bigg[ u_1 \sum_{i=2}^{3} u_i \partial_{y_1} u_i 
		+ \frac{1 + |y'|^2}{2y_1} u_2 \sum_{i \neq 2} u_i \partial_{y_2} u_i 
		+ \frac{1 + |y'|^2}{2y_1} u_3 \sum_{i \neq 3} u_i \partial_{y_3} u_i \Bigg],
	\end{aligned}
\end{equation}
where the relativistic local sound speed $c_s$ is given by
\begin{equation}\label{rel-sound-speed-B}
	c_s^2 = \frac{\partial p}{\partial \rho} = \frac{\frac{\partial p }{\partial n}}{\frac{\partial \rho}{\partial n}} = (\gamma - 1) \frac{\mathcal{B} \sqrt{c^2 - u^2} - c^2}{\mathcal{B} \sqrt{c^2 - u^2}} \, c^2 .
\end{equation}

The vorticity equations, along with \eqref{rel-deform-6}, establish a deformation-curl system for the velocity field
\begin{equation}\label{rel-deform-7}
	\begin{cases}
		\left( c^2 (c_s^2 - u_1^2) - c_s^2(u_2^2 + u_3^2)  \right) \partial_{y_1} u_1 
		+ \frac{1 + |y'|^2}{2y_1} \left(c^2 (c_s^2 - u_2^2) - c_s^2 ( u_1^2 + u_3^2) \right)\partial_{y_2} u_2 \\
		\quad + \frac{1 + |y'|^2}{2y_1} \left(c^2 (c_s^2 - u_3^2) - c_s^2 (u_1^2 + u_2^2) \right) \partial_{y_3} u_3 
		+ \frac{c_s^2 (c^2 - u^2)}{y_1} \left( 2u_1 - y_2 u_2 - y_3 u_3 \right) \\
		= (c^2 - c_s^2) \Bigg[ u_1 \sum_{i=2}^{3} u_i \partial_{y_1} u_i 
		+ \frac{1 + |y'|^2}{2y_1} u_2 \sum_{i \neq 2} u_i \partial_{y_2} u_i + \frac{1 + |y'|^2}{2y_1} u_3 \sum_{i \neq 3} u_i \partial_{y_3} u_i \Bigg],\\
		\frac{1 + |y'|^2}{2y_1} (\partial_{y_2} u_3 - \partial_{y_3} u_2) + \frac{1}{y_1} (y_3 u_2 - y_2 u_3) = \omega_1, \\[8pt]
		\frac{1 + |y'|^2}{2y_1} \partial_{y_3} u_1 - \partial_{y_1} u_3 - \frac{u_3}{y_1} = \omega_2, \\[8pt]
		\partial_{y_1} u_2 + \frac{u_2}{y_1} - \frac{1 + |y'|^2}{2y_1} \partial_{y_2} u_1 = \omega_3.
	\end{cases}
\end{equation}

\begin{lemma}\label{equiv}({\bf Equivalence.})
	{\it Assume that $C^1$ smooth vector functions $(n, {\bf u}, K)$ defined on a domain $\mathcal{Y}$ do not contain the vacuum (i.e. $n(y_1,y')>0$ in $\mathcal{Y}$) and the radial velocity $u_1(y_1,y')>0$ in $\mathcal{Y}$, then the following two statements are equivalent:
		\begin{enumerate}[(i)]
			\item $(n, {\bf u}, \mathcal{B})$ satisfy the steady Euler system \eqref{relativistic-euler-y} in $\mathcal{Y}$;
			\item $({\bf u}, K, \mathcal{B} )$ satisfy the equation \eqref{rel-deform-1},\eqref{rel-deform-2},\eqref{rel-deform-4} and \eqref{rel-deform-7}.
		\end{enumerate}
}\end{lemma}

Define 
\begin{equation*}
	\begin{aligned}
		& w_1(y_1, y')= u_1(y_1, y')- \bar{u}^+(y_1),\quad w_j(y_1, y') = u_j(y_1,y'), j=2,3,\\
		& w_4(y_1, y')=K(y_1, y')-\bar{K}^+,\quad w_5(y_1, y')= \mathcal{B} (y_1,y')-\bar{\mathcal{B} }^+,\\
		& w_6(y')= \xi(y')-r_s,\ \ \ {\bf w}(y_1, y')=(w_1,\cdots, w_5)(y_1, y').
	\end{aligned}
\end{equation*}
Then the density and the pressure become
\begin{equation}\label{n-p-w}
	\begin{aligned}
		n & = n(\mathbf{w}) = \left( \frac{\gamma - 1}{\gamma (\bar{K}^+ + w_4)} \right)^{\frac{1}{\gamma - 1}} 
		\Bigg( (\bar{\mathcal{B} }^+ + w_5) \sqrt{c^2 - (\bar{u}^+ + w_1)^2 - w_2^2 - w_3^2} - c^2 \Bigg)^{\frac{1}{\gamma - 1}} ,\\
		p & = p(\mathbf{w}) = \left( \frac{(\gamma - 1)^{\gamma}}{\gamma^{\gamma} (\bar{K}^+ + w_4)} \right)^{\frac{1}{\gamma - 1}} 
		\Bigg( (\bar{\mathcal{B} }^+ + w_5) \sqrt{c^2 - (\bar{u}^+ + w_1)^2 - w_2^2 - w_3^2} - c^2 \Bigg)^{\frac{\gamma}{\gamma - 1}}.
	\end{aligned}
\end{equation}

In terms of $w_1, w_2, \dots, w_5$, equations \eqref{rel-deform-1},\eqref{rel-deform-2}, \eqref{rel-deform-4}, and \eqref{rel-deform-omega1} take the following form. The hyperbolic quantities $w_4$ and $w_5$ satisfy
\begin{align}\label{rel-deform-w4}
	& \left( \partial_{y_1} + \frac{1 + |y'|^2}{2y_1 (\bar{u}^+ + w_1)} (w_2 \partial_{y_2} + w_3 \partial_{y_3}) \right) w_4 = 0,
	\\\label{rel-deform-w5}
	& \left( \partial_{y_1} + \frac{1 + |y'|^2}{2y_1 (\bar{u}^+ + w_1)} (w_2 \partial_{y_2} + w_3 \partial_{y_3}) \right) w_5 = 0.
\end{align}

The vorticity $\omega$ satisfies
\begin{equation}\label{rel-deform-omega1-w}
	\begin{aligned}
		&\partial_{y_1} \omega_1 + \frac{1 + |y'|^2}{2y_1 (\bar{u}^+ + w_1)} \sum_{i=2}^{3} w_i \partial_{y_i} \omega_1 
		+ \left[ \frac{1 + |y'|^2}{2y_1} \sum_{i=2}^{3} \partial_{y_i} \left( \frac{w_i}{\bar{u}^+ + w_1} \right) + \frac{2}{y_1} - \frac{y_2 w_2 + y_3 w_3}{y_1 (\bar{u}^+ + w_1)} \right] \omega_1 \\
		&\qquad + \frac{(1 + |y'|^2)^2}{4y_1^2} \Bigg[ 
		\partial_{y_2} \left( \frac{c^2 - (\bar{u}^+ + w_1)^2 - \sum_{i = 2}^3 w_i^2}{(\bar{u}^+ + w_1)(\bar{\mathcal{B} }^+ + w_5)} \right) \partial_{y_3} w_5 \\
		& \qquad
		- \partial_{y_3} \left( \frac{c^2 - (\bar{u}^+ + w_1)^2 - \sum_{i = 2}^3 w_i^2}{(\bar{u}^+ + w_1)(\bar{\mathcal{B} }^+ + w_5)} \right) \partial_{y_2} w_5 \\
		&\qquad - \partial_{y_2} \left( \frac{\sqrt{c^2 - (\bar{u}^+ + w_1)^2 - \sum_{i = 2}^3 w_i^2}}{\gamma (\bar{K}^+ + w_4) (\bar{u}^+ + w_1)(\bar{\mathcal{B} }^+ + w_5)} \right) \partial_{y_3} w_4 \\
		&\qquad + \partial_{y_3} \left( \frac{\sqrt{c^2 - (\bar{u}^+ + w_1)^2 - \sum_{i = 2}^3 w_i^2}}{\gamma (\bar{K}^+ + w_4) (\bar{u}^+ + w_1)(\bar{\mathcal{B} }^+ + w_5)} \iota (\mathbf{w}) \right) \partial_{y_2} w_4 \Bigg] = 0,
	\end{aligned}
\end{equation}
and 
\begin{equation}\label{rel-deform-4-w}
	\begin{aligned}
		\omega_2 &= \frac{w_2}{\bar{u}^+ + w_1} \omega_1 
		+ \frac{1 + |y'|^2}{2y_1 (\bar{u}^+ + w_1)} \Bigg( 
		\frac{c^2 - (\bar{u}^+ + w_1)^2 - \sum\nolimits_{i= 2}^3 w_i^2}{\bar{\mathcal{B} }^+ + w_5} \partial_{y_3} w_5 \\
		&\qquad - \frac{\sqrt{c^2 - (\bar{u}^+ + w_1)^2 - \sum\nolimits_{i= 2}^3 w_i^2}}{\gamma (\bar{K}^+ + w_4)(\bar{\mathcal{B} }^+ + w_5)}
		\iota (\mathbf{w}) \partial_{y_3} w_4 \Bigg), \\
		\omega_3 &= \frac{w_3}{\bar{u}^+ + w_1} \omega_1 
		+ \frac{1 + |y'|^2}{2y_1 (\bar{u}^+ + w_1)} \Bigg(
		\frac{c^2 - (\bar{u}^+ + w_1)^2 - \sum\nolimits_{i= 2}^3 w_i^2}{\bar{\mathcal{B} }^+ + w_5} \partial_{y_2} w_5 \\
		&\qquad - \frac{\sqrt{c^2 - (\bar{u}^+ + w_1)^2 - \sum\nolimits_{i= 2}^3 w_i^2}}{\gamma (\bar{K}^+ + w_4)(\bar{\mathcal{B} }^+ + w_5)}
		\iota (\mathbf{w}) \partial_{y_2} w_4 \Bigg),
	\end{aligned}
\end{equation}
with $\iota (\mathbf{w}) = (\bar{\mathcal{B} }^+ + w_5)\sqrt{c^2 - (\bar{u}^+ + w_1)^2 - \sum\nolimits_{i = 2}^3 w_i^2} - c^2$.

We now derive the equation satisfied by the velocity field. Since
\begin{align*}
	& c^2 \left(\bar{c}_s^2 (\bar{\mathcal{B}}^+, (\bar{u}^+)^2) - (\bar{u}^+)^2 \right) (\bar{u}^+)' (y_1) + \frac{2 \bar{c}_s^2 (\bar{\mathcal{B}}^+, (\bar{u}^+)^2) \left(c^2 - (\bar{u}^+)^2 \right) \bar{u}^+(y_1)}{y_1} = 0, \\
	& c_s^2 (\mathcal{B}, u^2) - u^2 - \bar{c}_s^2 (\bar{\mathcal{B}}^+, (\bar{u}^+)^2) + (\bar{u}^+)^2 \\ 
	& = - \left( 1 + \frac{(\gamma-1) c^4}{2 \bar{\mathcal{B}}^+ (c^2 - (\bar{u}^+)^2)^{3/2}} \right) 2 \bar{u}^+ w_1
	+ \frac{(\gamma-1) c^4}{(\bar{\mathcal{B}}^+)^2 \sqrt{c^2 - (\bar{u}^+)^2}} w_5 + \mathcal{R}_1,
\end{align*}
where
\begin{align*}
	\mathcal{R}_1 = & c_s^2 (\mathcal{B}, u^2) - \bar{c}_s^2 (\bar{\mathcal{B}}^+, (\bar{u}^+)^2) 
	- \left( 1 + \frac{(\gamma-1) c^4}{2 \bar{\mathcal{B}}^+ (c^2 - (\bar{u}^+)^2)^{3/2}} \right) \sum_{i = 1}^3 w_i^2 \\
	& 
	- (\gamma-1) c^4
	\left[
	\frac{w_5}{(\bar{\mathcal{B}}^+)^2 \sqrt{c^2 - (\bar{u}^+)^2}}
	- \frac{\bar{u}^+ w_1}{\bar{\mathcal{B}}^+ (c^2 - (\bar{u}^+)^2)^{3/2}}
	\right]
\end{align*} is the high-order term and 
\begin{equation*}
	\bar{c}_s^2 (\bar{\mathcal{B}}^+, (\bar{u}^+)^2) = (\gamma - 1) c^2 \left( 1 - \frac{ c^2}{\bar{\mathcal{B}}^+ \sqrt{c^2 - (\bar{u}^+)^2}} \right).
\end{equation*}
Then, if follows from the first equation of \eqref{rel-deform-7} that
\begin{equation}\label{rel-deform-w1}
	d_1(y_1) \partial_{y_1} w_1 + \frac{1 + |y'|^2}{2y_1} \sum_{j=2}^{3} \partial_{y_j} w_j 
	+ \frac{2w_1}{y_1} - \frac{1}{y_1} \sum_{j=2}^{3} y_j w_j + d_2(y_1) w_1 
	=d_0(y_1) w_5 + \mathcal{F}(\mathbf{w}),
\end{equation}
where
\begin{align*}
	& d_1(y_1) = 1 - \bar{\mathcal{M}}^2(y_1), \qquad \bar{\mathcal{M}}^2(y_1) = \frac{(\bar{u}^+(y_1))^2}{c^2 - (\bar{u}^+(y_1))^2} \cdot \frac{c^2 - \bar{c}_s^2(y_1)}{\bar{c}_s^2(y_1)}, \\
	& d_2(y_1) = \frac{
		2 c^2}{y_1(1-\bar{\mathcal{M}}^2)} \left( \frac{2 \bar{\mathcal{M}}^2}{c^2 -  (\bar{u}^+)^2} + \frac{(\gamma - 1) c^4  (\bar{u}^+)^4}{\bar{\mathcal{B}}^+ \bar{c}_s^4 (c^2 - (\bar{u}^+)^2)^{5/2}} \right), \\
	& d_0(y_1) = - \frac{((\gamma - 1)c^2 -\bar{c}_s^2 )^2 \sqrt{c^2 - (\bar{u}^+)^2}}{(\gamma - 1) c^4} \left(c^2 (\bar{u}^+)' + \frac{2 (c^2 - (\bar{u}^+)^2) \bar{u}^+}{y_1} \right),\\
	& \bar{c}_s^2 (c^2 - (\bar{u}^+)^2) \mathcal{F}(\mathbf{w}) =  
	\left( 1 + \frac{(\gamma - 1)c^4}{2 \bar{\mathcal{B}}^+ (c^2 - (\bar{u}^+)^2)^{3/2}} \right) 2 \bar{u}^+ c^2 w_1 \partial_{y_1} w_1 
	- \frac{(\gamma - 1) c^6}{(\bar{\mathcal{B}}^+)^2 \sqrt{c^2 - (\bar{u}^+)^2}} w_5 \partial_{y_1} w_1  \\
	&  \quad \quad - \left( c^2 \mathcal{R}_1 -  c_s^2 ( w_2^2 + w_3^2) \right)(\partial_{y_1} w_1 + (\bar{u}^+)' )- (c_s^2 - \bar{c}_s^2) (\bar{u}^+ + w_1) \sum\nolimits_{i = 2}^3 w_i \partial_{y_1} w_i \\
	&  \quad \quad  - \frac{1 + |y'|^2}{2y_1} \left((c_s^2 - \bar{c}_s^2) (c^2 - (\bar{u}^+)^2) - \bar{c}_s^2 (2 \bar{u}^+ + w_1^2 + w_3^2) - c^2 w_2^2  \right) \partial_{y_2} w_2 \\
	& \quad \quad   - \frac{1 + |y'|^2}{2y_1} \left((c_s^2 - \bar{c}_s^2) (c^2 - (\bar{u}^+)^2) - \bar{c}_s^2 (2 \bar{u}^+ + w_1^2 + w_2^2) - c^2 w_3^2  \right) \partial_{y_3} w_3 \\
	& \quad \quad  + \left( \left( 2 \bar{c}_s^2 + \frac{(\gamma - 1)c^4}{\bar{\mathcal{B}}^+ \sqrt{c^2 - (\bar{u}^+)^2}} \right) \frac{\bar{u}^+ w_1}{y_1} - \frac{(\gamma - 1)c^4 \sqrt{c^2 - (\bar{u}^+)^2}}{y_1 (\bar{\mathcal{B}}^+)^2} w_5  \right) (2 w_1 - y_2 w_2 - y_3 w_3)\\
	& \quad \quad
	- \left((c_s^2 - \bar{c}_s^2) 2\bar{u}^+  w_1 + c_s^2 \sum\nolimits_{i= 1}^3 w_i^2 - (c^2 - (\bar{u}^+)^2) (\mathcal{R}_1 + \sum\nolimits_{i= 1}^3 w_i^2) \sum\nolimits_{i= 1}^3 w_i^2  \right) \frac{ 2 \bar{u}^+ + 2 w_1 - y_2 w_2 - y_3 w_3}{y_1} \\
	& \quad \quad
	- \frac{1 + |y'|^2}{2y_1} (c_s^2 - \bar{c}_s^2) \left( w_2 \left( (\bar{u}^+ + w_1)\partial_{y_2} w_1 + w_3 \partial_{y_2} w_3 \right) + w_3 \left( (\bar{u}^+ + w_1)\partial_{y_3} w_1 + w_2 \partial_{y_3} w_2 \right)\right).
\end{align*}

\subsection{The reformulation of the Rankine-Hugoniot conditions and boundary conditions}\label{22}

It follows from the third and fourth equations in \eqref{RH} that 
\begin{equation}\label{RH-1}
	\partial_{y_2}\xi (y') = \frac{2\xi(y')}{1 + |y'|^2} \cdot \frac{ \mathcal{N}_2(\xi, y') }{ \mathcal{N}(\xi, y') }, \quad
	\partial_{y_3}\xi (y') = \frac{2\xi(y')}{1 + |y'|^2} \cdot \frac{ \mathcal{N}_3(\xi, y') }{ \mathcal{N}(\xi, y') },
\end{equation}
where
\begin{align*}
	\mathcal{N}(\xi, y') &= \left[ \dfrac{(p + \rho c^2) u_2^2}{c^2 - u^2} + p \right] \left[ \dfrac{(p + \rho c^2) u_3^2}{c^2 - u^2} + p \right] - \left(\left[ \dfrac{(p + \rho c^2) u_2 u_3}{c^2 - u^2} \right]\right)^2, \\
	\mathcal{N}_2(\xi, y') &= \left[ \dfrac{(p + \rho c^2) u_1 u_2}{c^2 - u^2} \right] \left[ \dfrac{(p + \rho c^2) u_3^2}{c^2 - u^2} + p \right] - \left[ \dfrac{(p + \rho c^2) u_1 u_3}{c^2 - u^2} \right] \left[ \dfrac{(p + \rho c^2) u_2 u_3}{c^2 - u^2} \right], \\
	\mathcal{N}_3(\xi, y') &= \left[ \dfrac{(p + \rho c^2) u_1 u_3}{c^2 - u^2} \right] \left[ \dfrac{(p + \rho c^2) u_2^2}{c^2 - u^2} + p \right] - \left[ \dfrac{(p + \rho c^2) u_1 u_2}{c^2 - u^2} \right] \left[ \dfrac{(p + \rho c^2) u_2 u_3}{c^2 - u^2} \right].
\end{align*}

We can rewrite \eqref{RH-1} that 
\begin{equation}\label{RH-2}
	\begin{aligned}
		\partial_{y_2}\xi (y') & = \frac{2}{1 + |y'|^2} \left(r_s \alpha_0 w_2 (\xi(y'), y') + r_s g_2 \right) , \\
		\partial_{y_3}\xi (y') & = \frac{2}{1 + |y'|^2} \left(r_s \alpha_0 w_3 (\xi(y'), y') + r_s g_3 \right),
	\end{aligned}
\end{equation}
where $\alpha_0 = \frac{(\bar{p}^+(r_s) + \bar{\rho}^+(r_s) c^2) \bar{u}^+ (r_s)}{(c^2 - (\bar{u}^+ (r_s))^2 ) [\bar{p}(r_s)]} > 0 $, and 
\begin{equation*}
	g_i = \frac{1}{r_s} \left((r_s + w_6 (y')) \frac{\mathcal{N}_i}{\mathcal{N}} (r_s + w_6 (y'), y') - \alpha_0 r_s w_i (\xi(y'), y') \right), \quad i = 2, 3.
\end{equation*}
Moreover, the functions $g_i$ $(i = 2,3)$ are the error terms and are bounded by 
\begin{equation}\label{g-bdd}
	|g_i| \leq C_1 \left(
	|\mathbf{w} (\xi, y')|^2 + |w_6(y')|^2 \right), \quad i = 2, 3.
\end{equation}

It follows from \eqref{RH} and \eqref{RH-1} that
\begin{equation}\label{RH-3} \begin{cases} \left[ \dfrac{n u_1}{\sqrt{c^2 - u^2}} \right] = \dfrac{1}{\mathcal{N}} \displaystyle\sum_{i=2}^3 \mathcal{N}_i \left[ \dfrac{n u_i}{\sqrt{c^2 - u^2}} \right], \\
		\left[ \dfrac{(p + \rho c^2) u_1^2}{c^2 - u^2} + p \right] = \frac{\mathcal{B} u_1^+}{\mathcal{N}} \sum_{i=2}^3 \mathcal{N}_i \left[ \frac{n u_i}{\sqrt{c^2 - u^2}} \right] + \frac{\mathcal{B} [u_1]}{\mathcal{N}} \sum_{i=2}^3 \mathcal{N}_i \left( \frac{n u_i}{\sqrt{c^2 - u^2}} \right)^-, \\ 
		[\mathcal{B}] = 0. \end{cases} \end{equation}

We observe that
\begin{equation*}
	\left[ \frac{\bar{n} \bar{u}}{\sqrt{c^2 - \bar{u}^2}} \right](r_s + w_6) = O (w_6^2), \quad
	\left[ \frac{(\bar{p} + \bar{\rho} c^2) \bar{u}^2}{c^2 - \bar{u}^2} + \bar{p} \right] (r_s + w_6) = \frac{2}{r_s} [\bar{p}] (r_s) w_6 + O(w_6^2).
\end{equation*}
Then, one can use the Taylor's expansion and \eqref{RH-3} to obtain that 
\begin{equation}\label{RH-4}
	\begin{cases}
		a_{11} w_1 + a_{12} w_4 = R_{01} (
		\mathbf{w}, w_6),  \\
		a_{21} w_1 + a_{22} w_4 = - \frac{2}{r_s} [\bar{p}] (r_s) w_6 + R_{02} (
		\mathbf{w}, w_6), 
	\end{cases}
\end{equation}
where
\begin{equation*}
	\begin{aligned}
		& a_{11} = \frac{\bar{n}^+(r_s)}{\sqrt{c^2 - (\bar{u}^+(r_s))^2}} (1 - \bar{\mathcal{M}}^2(r_s)), \\ 
		& a_{12} = - \frac{(\bar{n}^+ \bar{u}^+)(r_s)}{(\gamma-1) \bar{K}^+ \sqrt{c^2 - (\bar{u}^+(r_s))^2}}, \\
		& a_{21} = \frac{(\bar{p}^+ + \bar{\rho}^+ c^2)(r_s) \bar{u}^+(r_s)}{c^2 - (\bar{u}^+(r_s))^2} (1 - \bar{\mathcal{M}}^2(r_s)), \\
		& a_{22} = - \frac{(\bar{p}^+ + \bar{\rho}^+ c^2)(r_s) (\bar{u}^+(r_s))^2}{(\gamma-1) \bar{K}^+ (c^2 - (\bar{u}^+(r_s))^2)} - \dfrac{(\bar{n}^+(r_s))^\gamma}{\gamma-1},
	\end{aligned}
\end{equation*}
and 
\begin{align*}
	R_{01} = & \dfrac{1}{\mathcal{N}} \displaystyle\sum_{i=2}^3 \mathcal{N}_i \left[ \dfrac{n u_i}{\sqrt{c^2 - u^2}} \right] - \frac{n(\mathbf{w}) (\bar{u}^+ + w_1)}{\sqrt{c^2 - (\bar{u}^+ + w_1)^2 - w_2^2 - w_3^2}} + \frac{n^- u_1^-}{\sqrt{c^2 - (u^-)^2}} (\xi, y') \\
	& 
	- \frac{\bar{n}^- \bar{u}^-}{\sqrt{c^2 - (\bar{u}^-)^2}}(\xi) - \left[ \frac{\bar{n} \bar{u}}{\sqrt{c^2 - \bar{u}^2}} \right](\xi) + \frac{\bar{n}^+ \bar{u}^+}{\sqrt{c^2 - (\bar{u}^+)^2}} (\xi) + a_{11} w_1 + a_{12} w_4, \\
	R_{02} = &  \frac{\mathcal{B} u_1^+}{\mathcal{N}} \sum_{i=2}^3 \mathcal{N}_i \left[ \frac{n u_i}{\sqrt{c^2 - u^2}} \right] + \frac{\mathcal{B} [u_1]}{\mathcal{N}} \sum_{i=2}^3 \mathcal{N}_i \left( \frac{n u_i}{\sqrt{c^2 - u^2}} \right)^- \\
	& 
	- \frac{(p(\mathbf{w}) + \rho(\mathbf{w}) c^2) (\bar{u}^+ + w_1)^2}{c^2 - (\bar{u}^+ + w_1)^2 - w_2^2 - w_3^2} - p (\mathbf{w}) + \left(\frac{(p^- + \rho^- c^2) (u^-)^2}{c^2 - (u^-)^2} + p^- \right) (\xi, y') \\
	&
	- \left(\frac{(\bar{p}^- + \bar{\rho}^- c^2) (\bar{u}^-)^2}{c^2 - (\bar{u}^-)^2} + \bar{p}^- \right)(\xi) - \left( \left[ \frac{(\bar{p} + \bar{\rho} c^2) \bar{u}^2}{c^2 - \bar{u}^2} + \bar{p} \right] (\xi) - \frac{2}{r_s} [\bar{p}] (r_s) w_6 \right) \\
	& + \left(\frac{(\bar{p}^+ + \bar{\rho}^+ c^2) (\bar{u}^+)^2}{c^2 - (\bar{u}^+)^2} + \bar{p}^+ \right)(\xi) + a_{21} w_1 + a_{22} w_4.
\end{align*}
Solving \eqref{RH-4} yields 
\begin{equation}\label{RH-5}
	\begin{cases}
		w_1 = a_1 w_6 + R_1(
		\mathbf{w}, w_6) ,  \\
		w_4 = a_2 w_6 + R_2 (
		\mathbf{w}, w_6), \\
		w_5 = \mathcal{B}^- (r_s + w_6, y') - \bar{\mathcal{B}}^-, 
	\end{cases}
\end{equation}
with
\begin{equation*}
	\begin{aligned}
		a_1 & = \frac{2 \bar{u}^+ (r_s) [\bar{p}(r_s)]}{r_s \bar{p}^+ (r_s) (1 - \bar{\mathcal{M}}^2(r_s))} > 0, &
		a_2 & = \frac{2 (\gamma - 1) [\bar{p}(r_s)]}{r_s (\bar{n}^+ (r_s))^{\gamma}} > 0, \\
		R_1 & = \frac{a_{22} R_{01} - a_{12} R_{02}}{a_{11} a_{22} - a_{21} a_{12}} \eqqcolon \sum_{i = 1}^2 b_{1i} R_{0i}, &
		R_2 & = \frac{- a_{21} R_{01} + a_{11} R_{02}}{a_{11} a_{22} - a_{21} a_{12}} \eqqcolon \sum_{i = 1}^2 b_{2i} R_{0i}.
	\end{aligned}
\end{equation*}
One can prove that $R_i$ $(i = 1,2)$ are bounded by
\begin{equation}
	|R_{i}| \leq C_* \left(
	|\mathbf{w} (\xi, y')|^2 + |w_6(y')|^2 \right), \quad i = 1,2,
\end{equation}
where $C_* > 0$ depends only on the background solutions.

It follows from \eqref{n-p-w} and the Taylor's expansion that
\begin{equation}\label{p_ex}
	\begin{aligned} 
		\epsilon p_{ex} (y') = & 
		-  \frac{ \bar{n}^+(r_2) \bar{u}^+(r_2) \bar{\mathcal{B}}^+}{\sqrt{c^2 - (\bar{u}^+(r_2))^2}} w_1 (r_2, y') - \frac{\bar{p}^+ (r_2)}{\gamma \bar{K}^+}  w_4 (r_2, y') \\
		& 
		+ \bar{n}^+ (r_2) \sqrt{c^2 - (\bar{u}^+ (r_2))^2} w_5(r_2, y') + E (\mathbf{w}(r_2, y')),
\end{aligned}\end{equation}
where
\begin{equation}\label{E}\begin{aligned} 
		& E (\mathbf{w}(r_2, y')) =  p (\mathbf{w}) - \bar{p}^+ (r_2) - \bar{n}^+ (r_2) \sqrt{c^2 - (\bar{u}^+ (r_2))^2} w_5(r_2, y') \\
		& \quad + \frac{ \bar{\mathcal{B}}^+ \bar{n}^+(r_2) \bar{u}^+(r_2)}{\sqrt{c^2 - (\bar{u}^+(r_2))^2}} w_1 (r_2, y') + \frac{\bar{p}^+ (r_2)}{(\gamma - 1) \bar{K}^+ }  w_4 (r_2, y') ,
\end{aligned}\end{equation}
and the error term $E$ is bounded by 
\begin{equation}\label{E-bdd}
	|E(\mathbf{w}(r_2, y'))| \leq C_* |\mathbf{w}(r_2, y')|^2.
\end{equation}
This, together with \eqref{p-perturbation}, implies that
\begin{equation}\label{r2-bdry}\begin{aligned}
		& \frac{ \bar{\mathcal{B}}^+}{\sqrt{c^2 - (\bar{u}^+(r_2))^2}} w_1 (r_2, y') + \frac{\bar{\mathcal{B}}^+ \sqrt{c^2 - (\bar{u}^+ (r_2))^2} - c^2}{\gamma \bar{K}^+ \bar{u}^+ (r_2)}  w_4 (r_2, y')  \\
		& = - \frac{\epsilon p_{ex} (y')}{\bar{n}^+(r_2) \bar{u}^+(r_2)} + \frac{\sqrt{c^2 - (\bar{u}^+ (r_2))^2} }{\bar{u}^+(r_2)} w_5(r_2, y') + \frac{E(\mathbf{w}(r_2, y'))}{\bar{n}^+(r_2) \bar{u}^+(r_2)}.
\end{aligned}\end{equation}

The boundary condition on $|y'| = 1$ is
\begin{equation}\label{wall-bdry-2}
	(y_2 w_2 + y_3 w_3) (y_1, y') = 0, \quad \forall y_1 \in [r_s + w_6 (y'), r_2], \quad |y'| = 1.
\end{equation}

Therefore, solving the system of equations \eqref{relativistic-euler-y} with boundary conditions \eqref{RH}, \eqref{inlet-per}, \eqref{wall-bdry} and \eqref{p-perturbation} is equivalent to seeking a scalar function $w_6$ defined on $E$ and a vector-valued function $\mathbf{w}$ defined on $\mathcal{Y}_6 \coloneqq \{(y_1, y')\mid r_s + w_6 (y') < y_1 < r_2, \, y' \in E \}$ that satisfy \eqref{rel-deform-w4}-\eqref{rel-deform-w1}, along with the boundary conditions \eqref{RH-2}, \eqref{RH-5}, \eqref{r2-bdry} and \eqref{wall-bdry-2}.

\subsection{Transform to a fixed boundary value problem}\label{23}
Considering that the shock front is a free boundary, we introduce the coordinate transformation
\begin{equation}\label{trans-z}
	z_1 = \frac{y_1 - \xi (y')}{r_2 - \xi (y')} (r_2 - r_s) + r_s, \quad z' = (z_2, z_3) = (y_2, y_3),
\end{equation}
to map it onto a fixed boundary. Under this transformation, we have
\begin{equation}\label{tran-z-2}
	\begin{cases} 
		y_1 = z_1 + \frac{r_2 - z_1}{r_2 - r_s} v_6(z') \eqqcolon D_0^{v_6}, \\
		\partial_{y_1} = \frac{r_2 - r_s}{r_2 - r_s - v_6} \partial_{z_1} \eqqcolon D_1^{v_6}, \\
		\partial_{y_2} = \partial_{z_2} - \frac{r_2 - z_1}{r_2 - r_s - v_6} \partial_{z_2} v_6 \partial_{z_1} \eqqcolon D_2^{v_6},\\
		\partial_{y_3} = \partial_{z_3} - \frac{r_2 - z_1}{r_2 - r_s - v_6} \partial_{z_3} v_6 \partial_{z_1} \eqqcolon D_3^{v_6}, 
	\end{cases}
\end{equation}
where we denote $v_6(z') = \xi (z') - r_s$. The domain $ \{(y_1, y')\mid r_s + w_6 (y') < y_1 < r_2, \, y' \in E \}$ becomes 
\begin{equation*}
	\mathcal{D} \coloneqq \{(z_1, z')\mid r_s  < z_1 < r_2, \, z' \in E \} \quad E \coloneqq  \{ z'\mid |z'| < 1 \}.
\end{equation*}
Denote
\begin{equation*}
	\begin{aligned}
		\Gamma_s & = \{(z_1, z_2, z_3) \mid z_1 = r_s, \, z' \in E \}, \\
		\Gamma_2 & = \{(z_1, z_2, z_3) \mid z_1 = r_2, \, z' \in E \},\\
		\Gamma_0 & = \{(z_1, z_2, z_3) \mid z_1 \in [r_s,r_2], \, |z'| = 1 \}.
	\end{aligned}
\end{equation*}
Let 
\begin{equation*}
	\begin{aligned}
		v_j (z) & = w_j (z_1 + \frac{r_2 - z_1}{r_2 - r_s} v_6(z'), z'), \, j = 1,2,3,4,5, \quad \mathbf{v} = (v_1, v_2, \cdots, v_5),\\	
		\tilde{\omega}_j (z) & = \omega_j (z_1 + \frac{r_2 - z_1}{r_2 - r_s} v_6(z'), z'), \, j = 1, 2,3.
	\end{aligned}
\end{equation*}
In the new coordinate system, the density $n $ and pressure $p $ in \eqref{n-p-w} are expressed as
\begin{equation}\label{n-p-z}
	\begin{aligned}
		\tilde{n} (\mathbf{v}, v_6) & = \left[ \frac{\gamma - 1}{\gamma (\bar{K}^+ + v_4)} \right]^{\frac{1}{\gamma - 1}} \left[ (\bar{\mathcal{B}}^+ + v_5) \sqrt{c^2 - (\bar{u}^+(D_0^{v_6}) + v_1)^2 - v_2^2 - v_3^2} - c^2 \right]^{\frac{1}{\gamma - 1}},\\
		\tilde{p} (\mathbf{v}, v_6) & = \left[ \frac{(\gamma - 1)^{\gamma}}{\gamma^{\gamma} (\bar{K}^+ + v_4)} \right]^{\frac{1}{\gamma - 1}} \left[ (\bar{\mathcal{B}}^+ + v_5) \sqrt{c^2 - (\bar{u}^+(D_0^{v_6}) + v_1)^2 - v_2^2 - v_3^2} - c^2 \right]^{\frac{\gamma}{\gamma - 1}}.
	\end{aligned}
\end{equation}
The boundary conditions \eqref{RH-2} on $z_1 = r_s$ are transformed into
\begin{equation}\label{RH-6} \begin{aligned}
		\partial_{z_2} v_6(z') &= \frac{2}{1+|z'|^2} \left( r_s \alpha_0 v_2(r_s, z') + r_s g_2(\mathbf{v}(r_s, z'), v_6 (z')) \right), \\
		\partial_{z_3} v_6(z') &= \frac{2}{1+|z'|^2} \left( r_s \alpha_0 v_3(r_s, z') + r_s g_3(\mathbf{v}(r_s, z'), v_6 (z')) \right),
\end{aligned} \end{equation}
with 
\begin{equation} 
	g_i = \frac{1}{r_s} \left(  \frac{ (r_s + v_6(z')) \mathcal{N}_i (\mathbf{v}(r_s, z'), v_6 (z'))}{\mathcal{N}(\mathbf{v}(r_s, z'), v_6 (z'))}  - \alpha_0 r_s v_i (r_s, z') \right), \quad i = 2, 3.
\end{equation}

It follows from the first equation of \eqref{RH-5} that
\begin{equation}\label{RH-7}
	v_6(z') = \frac{1}{a_1} \left( v_1(r_s, z') - R_1\left( \mathbf{v}(r_s, z'), v_6(z') \right) \right) ,
\end{equation}
where $R_1 \left( \mathbf{v}(r_s, z'), v_6(z') \right)\coloneqq \sum_{i = 1}^2 b_{1i} R_{0i}\left( \mathbf{v}(r_s, z'), v_6(z') \right)$.

The boundary value problems for the entropy and the Bernoulli quantity in the $z-$coordinate system are derived from \eqref{rel-deform-w4}, \eqref{rel-deform-w5} and \eqref{RH-5} as follows
\begin{equation}\label{v4}
	\begin{cases}
		\left( D_1^{v_6} + \frac{1 + |z'|^2}{2 D_0^{v_6} (\bar{u}^+ ( D_0^{v_6}) + v_1)} \sum_{i=2}^3 v_i D_i^{v_6} \right) v_4 = 0, \\
		v_4 (r_s, z') = a_2 v_6 (z') + R_2(\mathbf{v} (r_s, z'), v_6 (z')),
	\end{cases}
\end{equation}
and
\begin{equation}\label{v5}
	\begin{cases}
		\left( D_1^{v_6} + \frac{1 + |z'|^2}{2 D_0^{v_6} (\bar{u}^+ ( D_0^{v_6}) + v_1)} \sum_{i=2}^3 v_i D_i^{v_6} \right) v_5 = 0, \\
		v_5 (r_s, z') = \mathcal{B}^- (r_s + v_6 (z'), z') - \bar{\mathcal{B}}^-,
	\end{cases}
\end{equation}
where $R_2(\mathbf{v} (r_s, z'), v_6 (z')) \coloneqq \sum_{i = 1}^2 b_{2i} R_{0i}(\mathbf{v} (r_s, z'), v_6 (z')) $.

Furthermore, the vorticity equation \eqref{rel-deform-omega1-w} is written as
\begin{equation}\label{omega-1}
	D_1^{v_6} \tilde{\omega}_1 + \frac{1 + |z'|^2}{2 D_0^{v_6} (\bar{u}^+ ( D_0^{v_6}) + v_1)} \sum_{i=2}^{3} v_i D_i^{v_6} \tilde{\omega}_1 
	+ \mu (\mathbf{v}, v_6) \tilde{\omega}_1 = J (\mathbf{v}, v_6),
\end{equation}
where
\begin{equation*}
	\begin{aligned}
		\mu (\mathbf{v}, v_6) = & \frac{1 + |z'|^2}{2 D_0^{v_6}} \sum_{i=2}^{3} D_i^{v_6} \left( \frac{v_i}{\bar{u}^+(D_0^{v_6}) + v_1} \right) + \frac{2}{D_0^{v_6}} - \frac{z_2 v_2 + z_3 v_3}{D_0^{v_6} (\bar{u}^+(D_0^{v_6}) + v_1)}, \\
		J (\mathbf{v}, v_6) = & - \frac{(1 + |z'|^2)^2}{4 (D_0^{v_6})^2} \Bigg[ 
		D_2^{v_6} \left( \frac{c^2 - (\bar{u}^+(D_0^{v_6}) + v_1)^2 - \sum_{i = 2}^3 v_i^2}{(\bar{u}^+(D_0^{v_6}) + v_1)(\bar{\mathcal{B}}^+ + v_5)} \right) D_3^{v_6} v_5 \\
		& 
		- D_3^{v_6} \left( \frac{c^2 - (\bar{u}^+(D_0^{v_6}) + v_1)^2 - \sum_{i = 2}^3 v_i^2}{(\bar{u}^+(D_0^{v_6}) + v_1)(\bar{\mathcal{B}}^+ + v_5)} \right) D_2^{v_6} v_5 \\
		& - D_2^{v_6} \left( \frac{\sqrt{c^2 - (\bar{u}^+(D_0^{v_6}) + v_1)^2 - \sum_{i = 2}^3 v_i^2}}{\gamma (\bar{K}^+ + v_4) (\bar{u}^+(D_0^{v_6}) + v_1)(\bar{\mathcal{B}}^+ + v_5)} \iota (\mathbf{v}) \right) D_3^{v_6} v_4 \\
		& + D_3^{v_6} \left( \frac{\sqrt{c^2 - (\bar{u}^+(D_0^{v_6}) + v_1)^2 - \sum_{i = 2}^3 v_i^2}}{\gamma (\bar{K}^+ + v_4) (\bar{u}^+(D_0^{v_6}) + v_1)(\bar{\mathcal{B}}^+ + v_5)} \iota (\mathbf{v}) \right) D_2^{v_6} v_4 \Bigg], \\
		\iota(\mathbf{v}) = & (\bar{\mathcal{B}}^+ + v_5)\sqrt{c^2 - (\bar{u}^+(D_0^{v_6}) + v_1)^2 - \sum\nolimits_{i = 2}^3 v_i^2} - c^2.
	\end{aligned}
\end{equation*}
In addition, the vorticity \eqref{rel-deform-4-w} takes the form of 
\begin{equation}\label{omega-23}
	\begin{aligned}
		\tilde{\omega}_2 & = \frac{1+|z'|^2}{2 D_0^{v_6}} D_3^{v_6} v_1- D_1^{v_6} v_3-\frac{v_3}{D_0^{v_6}} \\
		&= \frac{v_2}{\bar{u}^+(D_0^{v_6}) + v_1} \tilde{\omega}_1 
		+ \frac{1 + |z'|^2}{2 D_0^{v_6} (\bar{u}^+(D_0^{v_6}) + v_1)} \Bigg( \frac{c^2 - (\bar{u}^+(D_0^{v_6}) + v_1)^2 - \sum\nolimits_{i= 2}^3 v_i^2}{\bar{\mathcal{B}}^+ + v_5}  D_3^{v_6} v_5 \\
		&\qquad - \frac{\sqrt{c^2 - (\bar{u}^+(D_0^{v_6}) + v_1)^2 - \sum\nolimits_{i= 2}^3 v_i^2}}{\gamma (\bar{K}^+ + v_4)(\bar{\mathcal{B}}^+ + v_5)}
		\iota (\mathbf{v})  D_3^{v_6} v_4 \Bigg), \\
		\tilde{\omega}_3 & = D_1^{v_6} v_2+ \frac{v_2}{D_0^{v_6}}-\frac{1+|z'|^2}{2 D_0^{v_6}} D_2^{v_6} v_1 \\
		&= \frac{v_3}{\bar{u}^+(D_0^{v_6}) + v_1} \tilde{\omega}_1 
		- \frac{1 + |z'|^2}{2 D_0^{v_6} (\bar{u}^+(D_0^{v_6}) + v_1)} \Bigg(\frac{c^2 - (\bar{u}^+(D_0^{v_6}) + v_1)^2 - \sum\nolimits_{i= 2}^3 v_i^2}{\bar{\mathcal{B}}^+ + v_5}  D_2^{v_6} v_5 \\
		&\qquad - \frac{\sqrt{c^2 - (\bar{u}^+(D_0^{v_6}) + v_1)^2 - \sum\nolimits_{i= 2}^3 v_i^2}}{\gamma (\bar{K}^+ + v_4)(\bar{\mathcal{B}}^+ + v_5)}
		\iota (\mathbf{v})  D_2^{v_6} v_4 \Bigg).
	\end{aligned}
\end{equation}

Separating the principal part from the quadratic remainders, we arrive at
\begin{equation}\label{omega}
	\begin{aligned}
		& \frac{1 + |z'|^2}{2 z_1} (\partial_{z_2} v_3 - \partial_{z_3} v_2) + \frac{1}{z_1} (z_3 v_2 - z_2 v_3) = \tilde{\omega}_1 + H_1 (\mathbf{v}, v_6), \\
		& \frac{1 + |z'|^2}{2 z_1} \partial_{z_3} \left( v_1 + \frac{\bar{\mathcal{B}}^+ (c^2 - (\bar{u}^+)^2) - c^2 \sqrt{c^2 - (\bar{u}^+)^2}}{\gamma \bar{K}^+ \bar{\mathcal{B}}^+\bar{u}^+}  v_4 \right) - \partial_{z_1} v_3 - \frac{v_3}{z_1}  \\
		& \quad 
		=  \frac{1 + |z'|^2}{2 D_0^{v_6} (\bar{u}^+ (D_0^{v_6}) + v_1)} \frac{c^2 - (\bar{u}^+ (D_0^{v_6}) + v_1)^2 - \sum\nolimits_{i= 2}^3 v_i^2}{\bar{\mathcal{B}}^+ + v_5}  D_3^{v_6} v_5 \\
		& \qquad
		+ \frac{v_2}{\bar{u}^+ (D_0^{v_6}) + v_1} \tilde{\omega}_1 + H_2 (\mathbf{v}, v_6), \\
		& \partial_{z_1} v_2 + \frac{v_2}{z_1} - \frac{1 + |z'|^2}{2 z_1} \partial_{z_2} \left( v_1 +  \frac{\bar{\mathcal{B}}^+ (c^2 - (\bar{u}^+)^2) - c^2 \sqrt{c^2 - (\bar{u}^+)^2}}{\gamma \bar{K}^+ \bar{\mathcal{B}}^+ \bar{u}^+}  v_4 \right) \\
		& \quad 
		= - \frac{1 + |z'|^2}{2 D_0^{v_6} (\bar{u}^+ (D_0^{v_6}) + v_1)} \frac{c^2 - (\bar{u}^+ (D_0^{v_6}) + v_1)^2 - \sum\nolimits_{i= 2}^3 v_i^2}{\bar{\mathcal{B}}^+ + v_5}  D_2^{v_6} v_5 \\
		& \qquad + \frac{v_3}{\bar{u}^+ (D_0^{v_6}) + v_1} \tilde{\omega}_1 + H_3  (\mathbf{v}, v_6),
	\end{aligned}
\end{equation}
where
\begin{align*}
	H_1 (\mathbf{v}, v_6) = & - \frac{1 + |z'|^2}{2 } \left( \left(\frac{D_2^{v_6}}{D_0^{v_6}} - \frac{\partial_{z_2}}{z_1} \right) v_3 - \left(\frac{D_3^{v_6}}{D_0^{v_6}} - \frac{\partial_{z_3}}{z_1} \right) v_2 \right) + \frac{D_0^{v_6} - z_1}{z_1 D_0^{v_6}} (z_3 v_2 - z_2 v_3), \\
	H_2 (\mathbf{v}, v_6) = &  
	- \frac{1 + |z'|^2}{2 }\left(\frac{D_3^{v_6}}{D_0^{v_6}} - \frac{\partial_{z_3}}{z_1} \right) v_1 + (D_1^{v_6} - \partial_{z_1}) v_3 - \frac{D_0^{v_6} - z_1}{z_1 D_0^{v_6}} v_3  \\
	&
	- \frac{1 + |z'|^2}{2 D_0^{v_6} (\bar{u}^+ (D_0^{v_6}) + v_1)} \frac{\sqrt{c^2 - (\bar{u}^+ (D_0^{v_6}) + v_1)^2 - \sum\nolimits_{i= 2}^3 v_i^2}}{\gamma (\bar{K}^+ + v_4)(\bar{\mathcal{B}}^+ + v_5)} \iota (\mathbf{v})  D_3^{v_6} v_4 \\
	&
	+ \frac{1 + |z'|^2}{2 z_1} \frac{\bar{\mathcal{B}}^+ (c^2 - (\bar{u}^+)^2) - c^2 \sqrt{c^2 - (\bar{u}^+)^2}}{\gamma \bar{K}^+ \bar{\mathcal{B}}^+ \bar{u}^+} \partial_{z_3} v_4, \\
	H_3  (\mathbf{v}, v_6) = & 
	- (D_1^{v_6} - \partial_{z_1}) v_2 + \frac{D_0^{v_6} - z_1}{z_1 D_0^{v_6}} v_2 + \frac{1 + |z'|^2}{2} \left(\frac{D_2^{v_6}}{D_0^{v_6}} -\frac{\partial_{z_2}}{z_1} \right) v_1 \\
	&
	+ \frac{1 + |z'|^2}{2 D_0^{v_6} (\bar{u}^+ (D_0^{v_6}) + v_1)} \frac{\sqrt{c^2 - (\bar{u}^+ (D_0^{v_6}) + v_1)^2 - \sum\nolimits_{i= 2}^3 v_i^2}}{\gamma (\bar{K}^+ + v_4)(\bar{\mathcal{B}}^+ + v_5)} \iota (\mathbf{v})  D_2^{v_6} v_4 \\
	&
	- \frac{1 + |z'|^2}{2 z_1}  \frac{\bar{\mathcal{B}}^+ (c^2 - (\bar{u}^+)^2) - c^2 \sqrt{c^2 - (\bar{u}^+)^2}}{\gamma \bar{K}^+ \bar{\mathcal{B}}^+ \bar{u}^+} \partial_{z_2} v_4 .
\end{align*}

Next, we analyze the boundary conditions of vorticity at $z_1 = r_s$. The following reformulation of the jump conditions \eqref{RH-6} is crucial for solving the transonic shock problem. Note that \eqref{RH-6} is equivalent to that on $E$ there holds
\begin{equation}\label{RH-8}\begin{cases}
		F_2(z')\coloneqq\partial_{z_2} v_6- \frac{2 \alpha_0 r_s v_2(r_s, z')}{1+|z'|^2} -\frac{2}{1+|z'|^2}g_2({\bf v}(r_s,y'),v_6(z'))\equiv 0,\\
		F_3(z')\coloneqq\partial_{z_3} v_6- \frac{2 \alpha_0 r_s v_3(r_s, z')}{1+|z'|^2}-\frac{2}{1+|z'|^2}g_3({\bf v}(r_s,y'),v_6(z'))\equiv 0.
\end{cases}\end{equation}
To match the deformation-curl system for the velocity field, we reformulate \eqref{RH-8} by using the following equivalence lemma.
\begin{lemma}\label{equi0}\cite[Lemma 2.3]{Weng25}
	{\it Let $F_j \, ( j=2,3)$ be two $C^1$ smooth functions defined on $\overline{E}$. Then the following two statements are equivalent
		\begin{enumerate}[(i)]
			\item $F_2=F_3\equiv 0$ on $\overline{E}$;
			\item $F_2$ and $F_3$ solve the following problem
			\begin{equation}\label{equi00}\begin{cases}
					\partial_{z_2}F_3-\partial_{z_3} F_2=0,\ \ &\text{in}\ E,\\
					\partial_{z_2} F_2 + \partial_{z_3} F_3=0,\ \ &\text{in}\ E,\\
					z_2 F_2+z_3 F_3=0,\ \ &\text{on}\ \ z_2^2+z_3^2=1.
			\end{cases}\end{equation}
		\end{enumerate}
}\end{lemma}

Combining the first equation in \eqref{equi00} with \eqref{RH-8}, we obtain
\begin{equation}\label{bdry-1}
	\begin{aligned}
		& \frac{1 + |z'|^2}{2 r_s} \left(\partial_{z_2} v_3 - \partial_{z_3} v_2 \right) (r_s, z') + \frac{1}{r_s} \left(z_3 v_2 - z_2 v_3 \right) (r_s, z') \\
		& = \frac{(1 + |z'|^2)^2}{2 \alpha_0 r_s^2} \left(\partial_{z_3} \left( \frac{g_2({\bf v}(r_s,y'),v_6(z'))}{1 + |z'|^2}\right)  - \partial_{z_2} \left( \frac{g_3({\bf v}(r_s,y'),v_6(z'))}{1 + |z'|^2}\right)\right).
	\end{aligned} 
\end{equation}
Moreover, one has 
\begin{equation}\label{bdry-omega1}
	\begin{aligned}
		\tilde{\omega}_1 (r_s, z') = & \frac{(1 + |z'|^2)^2}{2 \alpha_0 r_s^2} \left(\partial_{z_3} \left( \frac{g_2({\bf v}(r_s,y'),v_6(z'))}{1 + |z'|^2}\right)  - \partial_{z_2} \left( \frac{g_3({\bf v}(r_s,z'),v_6(z'))}{1 + |z'|^2}\right)\right)  \\
		& + f_1 ({\bf v}(r_s,z'),v_6(z')),
	\end{aligned}
\end{equation}
where
\begin{equation*}\begin{aligned}
		f_1 ({\bf v}(r_s,z'),v_6(z'))= & \frac{v_6 }{r_s (r_s + v_6)} (z_2 v_3 - z_3 v_2) - \frac{1+ |z'|^2}{2} \frac{v_6}{r_s (r_s + v_6)} (\partial_{z_2} v_3 - \partial_{z_3} v_2) \\
		& 
		+ \frac{1+ |z'|^2}{2} \frac{v_6}{r_s (r_s + v_6)} \frac{r_2 - r_s}{r_2 - r_s - v_6 } (\partial_{z_2} v_6 \partial_{z_1} v_3 - \partial_{z_3} v_6 \partial_{z_1} v_2).
\end{aligned}\end{equation*}

Combining the second equation in \eqref{equi00} with \eqref{RH-8}, we obtain
\begin{equation}
	(\partial_{z_2}^2 + \partial_{z_3}^2 ) v_6 (z') - 2 \alpha_0 r_s \sum_{i = 2}^{3} \partial_{z_i} \left(\frac{v_i (r_s, z')}{1 + |z'|^2} \right) = 2 \sum_{i = 2}^3 \partial_{z_i} \left( \frac{g_i ({\bf v}(r_s,z'),v_6(z'))}{1 + |z'|^2}\right).
\end{equation}
Substituting \eqref{RH-7} into the above gives
\begin{equation}\label{bdry-2}
	(\partial_{z_2}^2 + \partial_{z_3}^2) v_1 (r_s, z') - 2 \alpha_0 a_1 r_s \sum_{i = 2}^{3} \partial_{z_i} \left(\frac{v_i (r_s, z')}{1 + |z'|^2} \right) = q_1({\bf v}(r_s,z'),v_6(z')) ,
\end{equation}
with 
\begin{equation*}
	q_1 ({\bf v}(r_s,y'),v_6(z')) = 2 \sum_{i = 2}^3 \partial_{z_i} \left( \frac{g_i ({\bf v}(r_s,z'),v_6(z'))}{1 + |z'|^2}\right) + \sum_{i = 2}^3 \partial_{z_i}^2 R_1 ({\bf v}(r_s,z'),v_6(z')).
\end{equation*}

Substituting \eqref{RH-7} into the boundary condition of \eqref{equi00} yields
\begin{equation}\label{bdry-3}
	\begin{aligned}
		& \sum_{i = 2}^3 z_i \partial_{z_i} v_1 (r_s, z') - \alpha_0 a_1 r_s \sum_{i = 2}^3 z_i v_i (r_s, z') \\
		& = \sum_{i = 2}^3 z_i \partial_{z_i} R_1 ({\bf v}(r_s,z'),v_6(z')) + \alpha_1  \sum_{i = 2}^3 z_i \partial_{z_i} g_i ({\bf v}(r_s,z'),v_6(z')), \quad \text{on } |z'| = 1.
\end{aligned}\end{equation}

The boundary condition on the wall is 
\begin{equation}\label{wall-bdry-3}
	(z_2 v_2 + z_3 v_3) (z_1, z') = 0, \quad \forall z_1 \in [r_s , r_2], \quad |z'| = 1.
\end{equation}

It follows from \eqref{r2-bdry} that the boundary condition at exit is 
\begin{equation}\label{bdry-4}\begin{aligned} 
		&  v_1 (r_2, z') + \frac{\bar{\mathcal{B}}^+ \left(c^2 - (\bar{u}^+ (r_2))^2 \right) - c^2 \sqrt{c^2 - (\bar{u}^+ (r_2))^2}}{\gamma \bar{K}^+ \bar{\mathcal{B}}^+ \bar{u}^+ (r_2)}  v_4 (r_2, z')   \\
		& = - \frac{\sqrt{c^2 - (\bar{u}^+(r_2))^2}}{\bar{\mathcal{B}}^+ \bar{n}^+(r_2)  \bar{u}^+(r_2)} \epsilon p_{ex} (z') + \frac{c^2 - (\bar{u}^+ (r_2))^2 }{\bar{\mathcal{B}}^+ \bar{u}^+(r_2)}v_5(r_2, z') + \frac{\sqrt{c^2 - (\bar{u}^+(r_2))^2}}{\bar{\mathcal{B}}^+ \bar{n}^+(r_2)  \bar{u}^+(r_2)}E(\mathbf{v}(r_2, z')),
\end{aligned}\end{equation}
with 
\begin{equation*}\begin{aligned} 
		& E (\mathbf{v}(r_2, z')) =  \tilde{p} (\mathbf{v} (r_2, z')) - \bar{p}^+ (r_2) - \bar{n}^+ (r_2) \sqrt{c^2 - (\bar{u}^+ (r_2))^2} v_5(r_2, z') \\
		& \quad + \bar{n}^+(r_2) \bar{u}^+(r_2) \left(\frac{ \bar{\mathcal{B}}^+}{\sqrt{c^2 - (\bar{u}^+(r_2))^2}} v_1 (r_2, z') + \frac{\bar{\mathcal{B}}^+ \sqrt{c^2 - (\bar{u}^+ (r_2))^2} - c^2}{\gamma \bar{K}^+ \bar{u}^+ (r_2)}  v_4 (r_2, z') \right).
\end{aligned}\end{equation*}

It follows from \eqref{rel-deform-w1} that
\begin{equation}\label{rel-deform-v1}
	d_1 (z_1) \partial_{z_1} v_1 + \frac{1 + |z'|^2}{2 z_1} \sum_{j = 2}^3 \partial_{z_j} v_j + \frac{2 v_1}{z_1} - \frac{1}{z_1} \sum_{j = 2}^{3} z_j v_j + d_2 (z_1) v_1 = d_0 (z_1) v_5 + H (\mathbf{v}, v_6),
\end{equation}
where
\begin{align*}
	H (\mathbf{v}, v_6) = & \mathcal{F}  (\mathbf{v}, v_6) - ( d_1 (D_0^{v_6}) D_1^{v_6} v_1 - d_1 (z_1) \partial_{z_1} v_1  ) - \frac{1 + |z'|^2}2 \sum_{j = 2}^3 \bigg(\frac{1}{D_0^{v_6}} D_j^{v_6} - \frac{1}{z_j} \partial_{z_j} \bigg) v_j \\
	& 
	- \bigg(\frac{1}{D_0^{v_6}} - \frac{1}{z_1} \bigg) \bigg( 2 v_1 - \sum_{j = 2}^3 z_j v_j \bigg) - ( d_2 (D_0^{v_6}) - d_2 (z_1) ) v_1,
\end{align*}
and 
\begin{align*}
	& \bar{c}_s^2 (c^2 - (\bar{u}^+ (D_0^{v_6}))^2)  \mathcal{F} (\mathbf{v}, v_6) = 
	\left( 1 + \frac{(\gamma - 1)c^4}{2 \bar{\mathcal{B}}^+ (c^2 - (\bar{u}^+(D_0^{v_6}))^2)^{3/2}} \right) 2 \bar{u}^+(D_0^{v_6}) c^2 v_1 D_1^{v_6} v_1 \\
	&\qquad 
	- \frac{(\gamma - 1) c^6 }{(\bar{\mathcal{B}}^+)^2 \sqrt{c^2 - (\bar{u}^+(D_0^{v_6}))^2}} v_5 D_1^{v_6} v_1  
	- \left( c^2 R_1 -  c_s^2 ( v_2^2 + v_3^2) \right)(D_1^{v_6} v_1 + (\bar{u}^+(D_0^{v_6}))' ) \\
	& \qquad
	- \frac{1 + |z'|^2}{2D_0^{v_6}} \left((c_s^2 - \bar{c}_s^2) (c^2 - (\bar{u}^+(D_0^{v_6}))^2) - \bar{c}_s^2 (2 \bar{u}^+(D_0^{v_6}) + v_1^2 + v_3^2) - c^2 v_2^2  \right) D_2^{v_6} v_2 \\
	& \qquad
	- \frac{1 + |z'|^2}{2D_0^{v_6}} \left((c_s^2 - \bar{c}_s^2) (c^2 - (\bar{u}^+(D_0^{v_6}))^2) - \bar{c}_s^2 (2 \bar{u}^+(D_0^{v_6}) + v_1^2 + v_2^2) - c^2 v_3^2  \right) D_3^{v_6} v_3 \\
	&\qquad 
	+ \left( \left( 2 \bar{c}_s^2 + \frac{(\gamma - 1)c^4}{\bar{\mathcal{B}}^+ \sqrt{c^2 - (\bar{u}^+(D_0^{v_6}))^2}} \right) \frac{\bar{u}^+(D_0^{v_6}) v_1}{D_0^{v_6}} - \frac{(\gamma - 1)c^4 \sqrt{c^2 - (\bar{u}^+(D_0^{v_6}))^2}}{D_0^{v_6} (\bar{\mathcal{B}}^+)^2} v_5  \right) \\
	& \qquad\quad 
	\times (2 v_1 - z_2 v_2 - z_3 v_3)\\
	&\qquad - \left((c_s^2 - \bar{c}_s^2) 2\bar{u}^+(D_0^{v_6})  v_1 + c_s^2 \sum\nolimits_{i= 1}^3 v_i^2 - (c^2 - (\bar{u}^+(D_0^{v_6}))^2) (\mathcal{R}_1 + \sum\nolimits_{i= 1}^3 v_i^2) \sum\nolimits_{i= 1}^3 v_i^2  \right) \\
	& \qquad\quad
	\times\frac{ 2 \bar{u}^+(D_0^{v_6}) + 2 v_1 - z_2 v_2 - z_3 v_3}{D_0^{v_6}} 
	- (c_s^2 - \bar{c}_s^2) (\bar{u}^+(D_0^{v_6}) + v_1) \sum\nolimits_{i = 2}^3 v_i D_1^{v_6} v_i\\
	&\qquad 
	- \frac{1 + |z'|^2}{2D_0^{v_6}} (c_s^2 - \bar{c}_s^2) \Big( v_2 \left( (\bar{u}^+(D_0^{v_6}) + v_1)D_2^{v_6} v_1 + v_3 D_2^{v_6} v_3 \right) \\
	& \qquad + v_3 \left( (\bar{u}^+(D_0^{v_6}) + v_1)D_3^{v_6} v_1 + v_2 D_3^{v_6} v_2 \right)\Big).
\end{align*}

To simplify the description of compatibility conditions on the shell boundary, we introduce polar coordinates
\begin{equation}\label{polar}
	y_2 = a \cos \tau, \quad y_3 = a \sin \tau,
\end{equation}
and define
\begin{equation}\label{u-a-tau}
	\begin{cases}
		u_a(y_1, a, \tau) \coloneqq u_2(y_1, a\cos\tau, a\sin\tau) \cos\tau + u_3(y_1, a\cos\tau, a\sin\tau) \sin\tau, \\
		u_{\tau} (y_1, a, \tau)\coloneqq -u_2(y_1, a \cos \tau, a \sin \tau) \sin \tau + u_3(y_1, a \cos \tau, a \sin \tau) \cos \tau,\\
		(u_1, \rho, p, K, \mathcal{B}) (y_1, a, \tau) \coloneqq (u_1, \rho, p, K, \mathcal{B}) (y_1, a\cos\tau, a\sin\tau).
	\end{cases}
\end{equation}

Under the coordinate $(y_1, a, \tau)$, the system \eqref{relativistic-euler-y} becomes
\begin{equation}\label{relativistic-euler-polar}
	\begin{cases}
		\partial_{y_1}\left( \frac{n u_1}{\sqrt{c^2 - u^2}} \right)
		+ \frac{1 + a^2}{2y_1} \left[ \partial_a \left( \frac{n u_a}{\sqrt{c^2 - u^2}} \right) 
		+ \frac{1}{a} \partial_\tau \left( \frac{n u_\tau}{\sqrt{c^2 - u^2}} \right) \right]  + \frac{2}{y_1} \frac{n u_1}{\sqrt{c^2 - u^2}}
		- \frac{a}{y_1} \frac{n u_a}{\sqrt{c^2 - u^2}} = 0, \\
		\left[ u_1 \partial_{y_1} 
		+ \frac{1 + a^2}{2y_1} \left( u_a \partial_a + \frac{u_\tau}{a} \partial_\tau \right) \right] u_1 
		- \frac{u_a^2 + u_\tau^2}{y_1}
		+ \frac{c^2 - u^2}{p + \rho c^2} \partial_{y_1} p = 0, \\
		\left[ u_1 \partial_{y_1} 
		+ \frac{1 + a^2}{2y_1} \left( u_a \partial_a + \frac{u_\tau}{a} \partial_\tau \right) \right] u_a
		+ \frac{u_1 u_a}{y_1} 
		- \frac{1 - a^2 }{2 a y_1 }  u_\tau^2 + \frac{c^2 - u^2}{p + \rho c^2} \frac{1 + a^2}{2y_1} \partial_a p = 0, \\
		\left[ u_1 \partial_{y_1} 
		+ \frac{1 + a^2}{2y_1} \left( u_a \partial_a + \frac{u_\tau}{a} \partial_\tau \right) \right] u_\tau
		+ \frac{u_1 u_\tau}{y_1} 
		+ \frac{1 - a^2 }{2 a y_1} u_\tau u_a  + \frac{c^2 - u^2}{p + \rho c^2} \frac{1 + a^2}{2y_1 a} \partial_\tau p = 0, \\
		\left[ u_1 \partial_{y_1} 
		+ \frac{1 + a^2}{2y_1} \left( u_a \partial_a + \frac{u_\tau}{a} \partial_\tau \right) \right] \mathcal{B} = 0.
	\end{cases}
\end{equation}

Therefore, after introducing the coordinate transformation \eqref{trans-z}, solving \eqref{relativistic-euler-y} with the boundary conditions \eqref{RH}, \eqref{inlet-per}, \eqref{wall-bdry} and \eqref{p-perturbation} is equivalent to solving the following problem.

\textbf{Problem P. } Find a scalar function $v_6$ defined on $E$ and a vector function $(v_1, v_2, \ldots, v_5)$ defined on $\mathcal{D}$ such that they satisfy the equations \eqref{v4}, \eqref{v5}, \eqref{omega} and \eqref{rel-deform-v1}, as well as the boundary conditions \eqref{RH-7}, \eqref{bdry-omega1}, \eqref{bdry-2}, \eqref{wall-bdry-3} and \eqref{bdry-4}.

Theorem \ref{existence} follows directly from the following theorem.

\begin{theorem}\label{main-2}
	Assume the compatibility condition \eqref{com-p} holds. There exists a small constant $\epsilon_0 > 0$, depending only on the background solution and the boundary data $p_{ex} \in C^{2, \alpha} (\bar{E})$, such that if $0 \leq \epsilon < \epsilon_0$, the problem \eqref{v4}, \eqref{v5}, \eqref{omega} and \eqref{rel-deform-v1}, together with the boundary conditions \eqref{RH-7}, \eqref{bdry-omega1}, \eqref{bdry-2}, \eqref{wall-bdry-3} and \eqref{bdry-4}, admits a unique solution $(v_1, v_2, v_3, v_4, v_5) (z)$ with the shock front $\mathcal{S}: z_1 = v_6(z')$ satisfying the following properties:
	\begin{enumerate}
		\item 
		The function $v_6 (z') \in C^{3, \alpha} (\bar{E})$ satisfies 
		\begin{equation}\label{est-1}
			\| v_6 (z') - r_s \|_{C^{3, \alpha} (\bar{E})} \leq C_* \epsilon,
		\end{equation}
		and 
		\begin{equation}\label{com-v6}
			\partial_a v_6 (1, \tau) = 0, \quad \forall \, \tau \in \mathbb{T}_{2\pi},
		\end{equation} 
		where $C_*$ is a positive constant depending only on the background solution, the supersonic incoming flow and the exit pressure.
		\item 
		The solution $(v_1, v_2, v_3, v_4, v_5) (z) \in (C^{2, \alpha}(\bar{\mathcal{D}}))^5$ satisfies
		\begin{equation}\label{est-2}
			\sum_{i = 1}^5 \| v_i \|_{C^{2, \alpha}(\bar{\mathcal{D}})} \leq C_* \epsilon,
		\end{equation}
		and the compatibility conditions
		\begin{equation}\label{com-v12345}
			\begin{cases}
				v_a (z_1, 1, \tau) = \left( \partial_a^2 v_a + \partial_a v_a \right)(z_1, 1, \tau) = 0,  & \forall \, (z_1, \tau) \in [r_s, r_2] \times \mathbb{T}_{2\pi}, \\
				\partial_a (v_1, v_{\tau}, v_4, v_5) (z_1, 1, \tau) = 0, & \forall \, (z_1, \tau) \in [r_s, r_2] \times \mathbb{T}_{2\pi}.
			\end{cases}
		\end{equation}
		Here we also ues the polar coordinate $z_2 = a \cos \tau$, $z_3 = a \sin \tau$, and the functions $v_a$, $v_{\tau}$ are defined as in \eqref{u-a-tau}.
	\end{enumerate}
\end{theorem}

\section{The iteration scheme and the proof of Theorem~\ref{main-2}}\label{proof}
We now prove Theorem~\ref{main-2}. Define the set $\Xi$ to consist of vector-valued functions $(v_1, v_2, v_3, v_4, v_5, v_6) \in (C^{2,\alpha}(\bar{\mathcal{D}}))^5 \times C^{3, \alpha} (\bar{E})$ satisfying the compatibility conditions
\begin{equation}\label{com-Xi}
	\begin{cases}
		v_a (z_1, 1, \tau) = \left( \partial_a^2 v_a + \partial_a v_a \right)(z_1, 1, \tau) = 0,  & \forall \, (z_1, \tau) \in [r_s, r_2] \times \mathbb{T}_{2\pi}, \\
		\partial_a (v_1, v_{\tau}, v_4, v_5) (z_1, 1, \tau) = 0, & \forall \, (z_1, \tau) \in [r_s, r_2] \times \mathbb{T}_{2\pi},\\
		\partial_a v_6 (1, \tau) = 0, & \forall \, \tau \in \mathbb{T}_{2\pi},
	\end{cases}
\end{equation}
and the estimate
\begin{equation}\label{est-Xi}
	\|({\bf v},v_6)\|_{\Xi}\coloneqq\sum_{j=1}^5 \|v_j\|_{C^{2,\alpha}(\overline{\mathcal{D}})}+\|v_6\|_{C^{3,\alpha}(\overline{E})}\leq \sigma_0,
\end{equation}
with $\sigma_0$ being a small positive constant to be determined later.

For any element $(\hat{\mathbf{v}}, \hat{v}_6)$ in $\Xi$, we construct a contraction mapping $\mathcal{P}$ that maps it to itself. The fixed point of this mapping is then the solution to \textbf{Problem P}. We observe that the deformation-curl system \eqref{omega} and \eqref{rel-deform-v1} contains hyperbolic quantities $v_4$ and $v_5$, and thus remains a hyperbolic-elliptic mixed system. Solving the transport equations via the method of characteristics yields explicit expressions for $v_4$ and $v_5$.
Moreover, we note that the expressions for $v_4$ and $v_5$ involve the shock front $v_6$. In view of \eqref{RH-7}, the principal term of $v_4$ is a scalar multiple of $v_1(r_s, \cdot)$, while $v_5$ may be treated as a higher-order term. Substituting these into \eqref{omega} and \eqref{rel-deform-v1}, we obtain a first-order elliptic system for $(v_1, v_2, v_3)$, which can be uniquely solved. Furthermore, $v_4$ and $v_6$ are also uniquely determined. The iteration scheme is given as follows.

\textbf{(i) } The shock front can be uniquely determined by 
\begin{equation}\label{RH-9}
	v_6(z') = \frac{1}{a_1} \left( v_1(r_s, z') - R_1\left( \hat{\mathbf{v}}(r_s, z'), \hat{v}_6(z') \right) \right), 
\end{equation}
assuming $v_1 (r_s, z')$ is known. 

\textbf{(ii) } We solve for the hyperbolic quantities, namely the entropy $v_4$ and the Bernoulli quantity $v_5$.

The function $v_5$ satisfies
\begin{equation}\label{v5-iteration}
	\begin{cases}
		\left( D_1^{\hat{v}_6} + \frac{1 + |z'|^2}{2 D_0^{\hat{v}_6} (\bar{u}^+ ( D_0^{\hat{v}_6}) + \hat{v}_1)} \sum_{i=2}^3 \hat{v}_i D_i^{\hat{v}_6} \right) v_5 = 0, \\
		v_5 (r_s, z') = \mathcal{B}^- (r_s + \hat{v}_6 (z'), z') - \bar{\mathcal{B}}^-.
	\end{cases}
\end{equation}

Using \eqref{tran-z-2}, the equation \eqref{v5-iteration} can be reformulated as follows
\begin{equation}\label{standard-transport}
	\partial_{z_1} v_5 + K_2(z_1, z') \partial_{z_2} v_5 + K_3(z_1, z') \partial_{z_3} v_5 = 0,
\end{equation}
where $K_i(z_1, z') \,(i = 2, 3)$ are defined by
\begin{equation*}
	K_i(z_1, z') \coloneqq \frac{(1 + |z'|^2)(r_2 - r_s - \hat{v}_6) \hat{v}_i}{2(r_2 - r_s) D_0^{\hat{v}_6} (\bar{u}^+ ( D_0^{\hat{v}_6}) + \hat{v}_1) + (1 + |z'|^2)(z_1 - r_2) \sum_{j=2}^3 \hat{v}_j \partial_{z_j} \hat{v}_6}.
\end{equation*}
The characteristic trajectory $\bar{z}'(t; z) = (\bar{z}_2(t; z), \bar{z}_3(t; z))$ is determined by the following system
\begin{equation}\label{characteristic-ode}
	\begin{cases}
		\frac{\d \bar{z}_2(t; z)}{\d t} = K_2(t, \bar{z}_2(t; z), \bar{z}_3(t; z)), \quad \forall t \in [r_s, r_2], \\
		\frac{\d\bar{z}_3(t; z)}{\d t} = K_3(t, \bar{z}_2(t; z), \bar{z}_3(t; z)), \quad \forall t \in [r_s, r_2], \\
		\bar{z}_2(z_1; z) = z_2, \quad \bar{z}_3(z_1; z) = z_3.
	\end{cases}
\end{equation}

To analyze the boundary compatibility conditions on the nozzle wall, we introduce the polar coordinates $z_2 = a \cos \tau$, $z_3 = a \sin \tau$. Define
\begin{equation}\label{Ka-Ktau-definition}
	\begin{cases}x
		K_a(z_1, a, \tau) \coloneqq \cos\tau K_2 + \sin\tau K_3 = \frac{(1 + a^2)(r_2 - r_s - \hat{v}_6) \hat{v}_a}{2(r_2 - r_s) D_0^{\hat{v}_6} (\bar{u}^+ ( D_0^{\hat{v}_6}) + \hat{v}_1) + (1 + a^2)(z_1 - r_2) \left( \hat{v}_a \partial_a \hat{v}_6 + \frac{1}{a} \partial_\tau \hat{v}_6 \right)}, \\
		K_\tau(z_1, a, \tau) \coloneqq -\sin\tau K_2 + \cos\tau K_3 = \frac{(1 + a^2)(r_2 - r_s - \hat{v}_6) \hat{v}_\tau}{2(r_2 - r_s) D_0^{\hat{v}_6} (\bar{u}^+ ( D_0^{\hat{v}_6}) + \hat{v}_1) + (1 + a^2)(z_1 - r_2) \left( \hat{v}_a \partial_a \hat{v}_6 + \frac{1}{a} \partial_\tau \hat{v}_6 \right)},
	\end{cases}
\end{equation}
where $\hat{v}_a = \hat{v}_2 \cos\tau + \hat{v}_3 \sin\tau$ and $\hat{v}_\tau = -\hat{v}_2 \sin\tau + \hat{v}_3 \cos\tau$. 
Writing the characteristic trajectory in polar coordinates as $\bar{z}'(t; z) = (A(t) \cos \vartheta(t), A(t) \sin \vartheta(t))$, then $A(t)$ satisfies
\begin{equation*}
	\frac{\d}{\d t} A(t; z_1, a, \tau) = K_a(t, A, \vartheta), \quad \forall t \in [r_s, r_2].
\end{equation*}
On the boundary $|z'| = 1$, the slip boundary condition $z_2 \hat{v}_2 + z_3 \hat{v}_3 = 0$ yields $K_a(z_1, 1, \tau) = 0$. By the uniqueness of solutions to ordinary differential equations, for any $z_1 \in [r_s, r_2]$ and $z_2^2 + z_3^2 = 1$, one has
\begin{equation}\label{A-1}
	A(t; z_1, 1, \tau) = \sqrt{\bar{z}_2^2(t; z_1, 1, \tau) + \bar{z}_3^2(t; z_1, 1, \tau)} = 1, \quad \forall (t, z_1, \tau) \in [r_s, r_2]^2 \times \mathbb{T}_{2\pi}.
\end{equation}

Set $\boldsymbol{\zeta}(z) = (\zeta_2(z), \zeta_3(z)) \coloneqq (\bar{z}_2(r_s; z), \bar{z}_3(r_s; z))$, then
\begin{equation}\label{zeta-boundary}
	\zeta_2^2(z_1, 1, \tau) + \zeta_3^2(z_1, 1, \tau) = 1, \quad (z_1, \tau) \in [r_s, r_2] \times \mathbb{T}_{2\pi}.
\end{equation}
It follows from \eqref{characteristic-ode} that
\begin{equation}\label{zeta-estimate}
	\sum_{j=2}^3 \|\zeta_j(z) - z_j\|_{C^{2, \alpha}(\bar{\mathcal{D}})} \leq C_* \|(\hat{\mathbf{v}}, \hat{v}_6)\|_{\Xi}.
\end{equation}

Since $v_5$ is conserved along the trajectory, one has
\begin{equation*}
	v_5(z)=v_5(r_s,\boldsymbol{\zeta}(z))=\mathcal{B}^-(r_s+\hat{v}_6(\boldsymbol{\zeta}(z)),\boldsymbol{\zeta}(z))-\bar{\mathcal{B}}^-,
\end{equation*}
and the following estimate holds
\begin{equation}\label{est-v5}
	\begin{aligned}
		\|v_5\|_{C^{2,\alpha}(\overline{\mathcal{D}})} & \leq C_*\epsilon+C_*\epsilon (\|\hat{v}_6\|_{C^{2,\alpha}(\overline{E})}+\sum_{j=2}^3\|\zeta_j-z_j\|_{C^{2,\alpha}(\overline{\mathcal{D}})})\\\no
		&\leq C_*(\epsilon+\epsilon\|(\hat{{\bf v}}, \hat{v}_6)\|_{\Xi})\leq C_*(\epsilon+\epsilon \sigma_0).
	\end{aligned}
\end{equation}

Next, we analyze the compatibility conditions for $v_5$. First, using polar coordinates and \eqref{Ka-Ktau-definition}, we rewrite \eqref{standard-transport} as the following boundary value problem
\begin{equation}\label{v5-polar-transport}
	\begin{cases}
		\partial_{z_1} v_5 + K_a(z_1, a, \tau) \partial_a v_5 + \dfrac{K_\tau(z_1, a, \tau)}{a} \partial_\tau v_5 = 0,\\
		v_5(r_s, a, \tau) =  \mathcal{B}^-\big(r_s + \hat{v}_6(a, \tau), a, \tau\big) - \bar{\mathcal{B}}^-.
	\end{cases}
\end{equation}
Differentiating \eqref{v5-polar-transport} with respect to $a$ and evaluating at $a = 1$ yields that $\partial_a v_5$ satisfies the following on the wall
\begin{equation*}
	\begin{cases}
		\left( \partial_{z_1} (\partial_a v_5) + \dfrac{K_\tau}{a} \partial_\tau (\partial_a v_5) + \partial_a K_a \partial_a v_5 \right)(z_1, 1, \tau) = 0, \quad \forall (z_1, \tau) \in [r_s, r_2] \times \mathbb{T}_{2\pi}, \\
		(\partial_a v_5)(r_s, 1, \tau) = \partial_{z_1} \mathcal{B}^-(r_s + \hat{v}_6, 1, \tau)\partial_a \hat{v}_6(1,\tau) + \partial_a \mathcal{B}^-(r_s + \hat{v}_6, 1, \tau) = 0, \quad \forall \tau \in \mathbb{T}_{2\pi}.
	\end{cases}
\end{equation*}
By the uniqueness of solutions to ordinary differential equations, we obtain
\begin{equation}\label{com-v5}
	\partial_a v_5(z_1, 1, \tau) \equiv 0, \quad \forall (z_1, \tau) \in [r_s, r_2] \times \mathbb{T}_{2\pi}.
\end{equation}

The function $v_4$ satisfies
\begin{equation}\label{v4-iteration}
	\begin{cases}
		\left( D_1^{\hat{v}_6} + \frac{1 + |z'|^2}{2 D_0^{\hat{v}_6} (\bar{u}^+ ( D_0^{\hat{v}_6}) + \hat{v}_1)} \sum_{i=2}^3 \hat{v}_i D_i^{\hat{v}_6} \right) v_4 = 0, \\
		v_4 (r_s, z') = a_2 v_6 (z') + R_2(\hat{\mathbf{v}} (r_s, z'), \hat{v}_6 (z')).
	\end{cases}
\end{equation}
Applying the method of characteristics together with \eqref{RH-9} yields
\begin{equation}\label{v4-exp}
	\begin{aligned}
		v_4 (z) & = v_4 (r_s,\boldsymbol{\zeta}(z)) 
		= a_2 v_6 (\boldsymbol{\zeta}(z)) + R_2 (\hat{\mathbf{v}} (r_s, \boldsymbol{\zeta}(z)), \hat{v}_6 (\boldsymbol{\zeta}(z)))\\
		& = a_2 v_6 (z') + a_2 (v_6 (\boldsymbol{\zeta}(z))  - v_6 (z')) + R_2 (\hat{\mathbf{v}} (r_s, \boldsymbol{\zeta}(z)), \hat{v}_6 (\boldsymbol{\zeta}(z)))\\
		& = 
		\frac{a_2}{a_1} v_1 (r_s, z') + a_2 (v_6 (\boldsymbol{\zeta}(z))  - v_6 (z')) + R_3 (\hat{\mathbf{v}} (r_s, \boldsymbol{\zeta}(z)), \hat{v}_6 (\boldsymbol{\zeta}(z))),
	\end{aligned}
\end{equation}
with 
\begin{equation*}
	R_3 (\hat{\mathbf{v}} (r_s, \boldsymbol{\zeta}(z)), \hat{v}_6 (\boldsymbol{\zeta}(z))) = R_2 (\hat{\mathbf{v}} (r_s, \boldsymbol{\zeta}(z)), \hat{v}_6 (\boldsymbol{\zeta}(z))) - \frac{a_2}{\alpha_1} R_1 ( \hat{\mathbf{v}} (r_s, z'), \hat{v}_6 (z') ).
\end{equation*}

Since $v_6$ remains unknown, we can denote \eqref{v4-exp} by
\begin{equation}\label{v4-exp-1}
	v_4 (z) = \frac{a_2}{a_1} v_1 (r_s, z') + R_4  (\hat{\mathbf{v}} (r_s, \boldsymbol{\zeta}(z)), \hat{v}_6 (\boldsymbol{\zeta}(z))),
\end{equation}
where
\begin{equation*}
	R_4  (\hat{\mathbf{v}} (r_s, \boldsymbol{\zeta}(z)), \hat{v}_6 (\boldsymbol{\zeta}(z))) = a_2 (v_6 (\boldsymbol{\zeta}(z))  - v_6 (z')) + R_3 (\hat{\mathbf{v}} (r_s, \boldsymbol{\zeta}(z)), \hat{v}_6 (\boldsymbol{\zeta}(z))).
\end{equation*}
Hence, $v_4$ satisfies the estimate
\begin{equation}\label{v4-est}
	\begin{aligned}
		\| v_4 \|_{C^{2, \alpha} (\bar{\mathcal{D}})}  \leq & C_* (\| v_1 (r_s, \cdot) \|_{C^{2, \alpha}(\bar{E})} + \| R_4 \|_{C^{2, \alpha} (\bar{\mathcal{D}})}) \\
		\leq & C_* ( \| v_1 (r_s, \cdot) \|_{C^{2, \alpha}(\bar{E})} + \|\hat{v}_6\|_{C^{2,\alpha}(\overline{E})} \sum_{j=2}^3\|\zeta_j-z_j\|_{C^{2,\alpha}(\overline{\mathcal{D}})}) \\ 
		& + C_* (\epsilon \|(\hat{{\bf v}}, \hat{v}_6)\|_{\Xi} + \|(\hat{{\bf v}}, \hat{v}_6)\|_{\Xi}^2) \\
		\leq & C_* \| v_1 (r_s, \cdot) \|_{C^{2, \alpha}(\bar{E})} + C_* (\epsilon \sigma_0 + \sigma_0^2).
	\end{aligned}
\end{equation}

By virtue of the compatibility conditions \eqref{com-Xi}, we can derive
\begin{equation}\label{com-N}
	\begin{cases}
		\mathcal{N}_a (\hat{{\bf v}}(r_s,1,\tau),\hat{v}_6(1,\tau)) =  0,   \\
		\partial_a (\mathcal{N}_1, \mathcal{N}_{\tau}) (\hat{{\bf v}}(r_s,1,\tau),\hat{v}_6(1,\tau)) = 0,\\
		\cos\tau \, g_2 (\hat{{\bf v}}(r_s,1,\tau),\hat{v}_6(1,\tau)) + \sin \tau \, g_3 (\hat{{\bf v}}(r_s,1,\tau),\hat{v}_6(1,\tau)) = 0, \\
		\partial_a R_{0i} (\hat{{\bf v}}(r_s,1,\tau),\hat{v}_6(1,\tau))= 0 , \, i = 1, 2, 3,
	\end{cases} \forall \, \tau \in \mathbb{T}_{2\pi},
\end{equation}
where $\mathcal{N}_a = \cos \tau \mathcal{N}_2 + \sin \tau \mathcal{N}_3$ and $\mathcal{N}_{\tau} = -\sin \tau \mathcal{N}_2 + \cos \tau \mathcal{N}_3$. Consequently, for $i = 1, 2, 3$, 
\begin{equation}\label{com-R}
	\partial_a R_i (\hat{{\bf v}}(r_s,1,\tau),\hat{v}_6(1,\tau))= 0, \quad \forall  \tau \in \mathbb{T}_{2\pi},
\end{equation}
which yields 
\begin{equation}\label{com-R4}
	\partial_a R_{4}(\hat{{\bf v}}(r_s,\boldsymbol{\zeta}),\hat{v}_6(\boldsymbol{\zeta})) (z_1,1,\tau)=0, \ \ \forall (z_1,\tau)\in [r_s,r_2]\times \mathbb{T}_{2\pi},
\end{equation}
and 
\begin{equation}\label{com-v4}
	\partial_a v_4(z_1,1,\tau)=\frac{a_2}{a_1}\partial_a v_1(r_s,1,\tau),\ \ \forall (z_1,\tau)\in [r_s,r_2]\times \mathbb{T}_{2\pi}.
\end{equation}

{\bf (iii) } We analyze the first component of the vorticity. It follows from \eqref{omega-1} and \eqref{bdry-omega1} that
\begin{equation}\label{omega1}
	\begin{cases}
		D_1^{\hat{v}_6} \tilde{\omega}_1 + \frac{1 + |z'|^2}{2 D_0^{\hat{v}_6} (\bar{u}^+ ( D_0^{\hat{v}_6}) + \hat{v}_1)} \sum_{i=2}^{3} \hat{v}_i D_i^{\hat{v}_6} \tilde{\omega}_1 
		+ \mu (\mathbf{\hat{v}}, \hat{v}_6) \tilde{\omega}_1 = J (\mathbf{\hat{v}}, \hat{v}_6), \\
		\tilde{\omega}_1 (r_s, z') = R_5 (\mathbf{\hat{v}}(r_s, z'), \hat{v}_6(z')), \quad z' \in \bar{E},
	\end{cases}
\end{equation}
where
\begin{equation*}\begin{aligned}
		& R_5 (\mathbf{\hat{v}}(r_s, z'), \hat{v}_6(z'))=  \frac{(1 + |z'|^2)^2}{2 \alpha_0 r_s^2} \bigg(\partial_{z_3} \left( \frac{g_2({\bf \hat{v}}(r_s,y'),\hat{v}_6(z'))}{1 + |z'|^2}\right) \\ 
		& \qquad \quad - \partial_{z_2} \left( \frac{g_3({\bf \hat{v}}(r_s,y'),\hat{v}_6(z'))}{1 + |z'|^2}\right)\bigg) 
		+ f_1 ({\bf \hat{v}}(r_s,y'),\hat{v}_6(z')).
\end{aligned}\end{equation*}
Integrating \eqref{omega1} along the characteristic trajectory $(\tau, \bar{z}_2(\tau; z), \bar{z}_3(\tau; z))$ and utilizing the boundary conditions, we directly obtain
\begin{equation} \label{explicit-omega1}
	\begin{aligned}
		\tilde{\omega}_1(z_1, z') =\; & R_5 \big(\mathbf{\hat{v}}(r_s, \boldsymbol{\zeta}(z)), \hat{v}_6(\boldsymbol{\zeta}(z))\big) \exp \left( - \int_{r_s}^{z_1} \frac{\mu(\tau, \bar{z}'(\tau; z))}{M(\tau, \bar{z}'(\tau; z))} d\tau \right) \\
		& + \int_{r_s}^{z_1} \frac{J(t, \bar{z}'(t; z))}{M(t, \bar{z}'(t; z))} \exp \left( - \int_{t}^{z_1} \frac{\mu(\tau, \bar{z}'(\tau; z))}{M(\tau, \bar{z}'(\tau; z))} d\tau \right) dt,
	\end{aligned}
\end{equation}
where
\begin{equation*}
	M(z_1, z') =  \frac{2(r_2 - r_s) D_0^{\hat{v}_6} (\bar{u}^+(D_0^{\hat{v}_6}) + \hat{v}_1) - (1 + |z'|^2)(r_2 - z_1) \sum_{j=2}^3 \hat{v}_j \partial_{z_j} \hat{v}_6}{2 D_0^{\hat{v}_6} (\bar{u}^+(D_0^{\hat{v}_6}) + \hat{v}_1) (r_2 - r_s - \hat{v}_6)}.
\end{equation*}
Consequently, we can obtain that $\tilde{\omega}_1$ satisfies the following estimate
\begin{equation}\label{estimate-omega1}
	\begin{aligned}
		\|\tilde{\omega}_1\|_{C^{1, \alpha} (\bar{\mathcal{D}})} & \le C_* \left( \|\tilde{\omega}_1(r_s, \cdot)\|_{C^{1, \alpha} (\bar{E})} +  \|J (\mathbf{\hat{v}}, \hat{v}_6)\|_{C^{1, \alpha} (\bar{\mathcal{D}})} \right) \\
		& \leq C_* (\epsilon \|(\hat{{\bf v}}, \hat{v}_6)\|_{\Xi} + \|(\hat{{\bf v}}, \hat{v}_6)\|_{\Xi}^2) \leq C_* (\epsilon \sigma_0 + \sigma_0^2).
	\end{aligned}
\end{equation}

It follows from \eqref{com-Xi} that  
\begin{equation}\label{omega1-com}\begin{aligned}
		& \tilde{\omega}_1 (r_s, 1, \tau) = R_5 (\mathbf{\hat{v}}(r_s, 1, \tau), \hat{v}_6(1, \tau )) = 0 ,\, \forall \tau \in \mathbb{T}_{2\pi},  \\
		& J (\hat{{\bf v}}, \hat{v}_6)(z_1, 1, \tau) = 0, \, \forall (z_1, \tau) \in [r_s, r_2] \times \mathbb{T}_{2\pi}.
\end{aligned}\end{equation}
Combining \eqref{omega1-com} with \eqref{A-1}, \eqref{zeta-boundary} and \eqref{explicit-omega1} yields the following compatibility condition
\begin{equation}\label{com-omega1}
	\tilde{\omega}_1 (z_1, 1, \tau) = 0, \quad \forall  (z_1, \tau) \in [r_s, r_2] \times \mathbb{T}_{2\pi}.
\end{equation}

Substituting \eqref{v4-exp-1} and \eqref{explicit-omega1} into \eqref{omega} yields
\begin{equation}\label{curl}
	\begin{aligned}
		& \frac{1 + |z'|^2}{2 z_1} (\partial_{z_2} v_3 - \partial_{z_3} v_2) + \frac{1}{z_1} (z_3 v_2 - z_2 v_3) = \tilde{H}_1 (\hat{\mathbf{v}}, \hat{v}_6) , \\
		& \frac{1 + |z'|^2}{2 z_1} \partial_{z_3}  v_1  - \partial_{z_1} v_3 - \frac{v_3}{z_1} + \frac{1 + |z'|^2}{2 z_1} d_3 (z_1) \partial_{z_3} v_1 (r_s, z')
		=  \tilde{H}_2 (v_5, \hat{\mathbf{v}}, \hat{v}_6), \\
		& \partial_{z_1} v_2 + \frac{v_2}{z_1} - \frac{1 + |z'|^2}{2 z_1} \partial_{z_2} v_1  - \frac{1 + |z'|^2}{2 z_1} d_3 (z_1) \partial_{z_2} v_1 (r_s, z') 
		=  \tilde{H}_3 (v_5, \hat{\mathbf{v}}, \hat{v}_6),
	\end{aligned}
\end{equation}
where 
\begin{align*}
		d_3 (z_1) & =  \frac{\bar{\mathcal{B}}^+ (c^2 - (\bar{u}^+(z_1))^2) - c^2 \sqrt{c^2 - (\bar{u}^+(z_1))^2}}{\gamma \bar{K}^+ \bar{\mathcal{B}}^+ \bar{u}^+ (z_1)} \frac{a_2}{a_1} > 0 , \\
		\tilde{H}_1 (\hat{\mathbf{v}}, \hat{v}_6)  & = \tilde{\omega}_1 + H_1 (\hat{\mathbf{v}}, \hat{v}_6),\\
		\tilde{H}_2 (v_5, \hat{\mathbf{v}}, \hat{v}_6) &= \frac{\hat{v}_2}{\bar{u}^+ (D_0^{\hat{v}_6}) + \hat{v}_1} \tilde{\omega}_1 
		+ \frac{1 + |z'|^2}{2 D_0^{\hat{v}_6} (\bar{u}^+ (D_0^{\hat{v}_6}) + \hat{v}_1)} \frac{c^2 - (\bar{u}^+ (D_0^{\hat{v}_6}) + \hat{v}_1)^2 - \sum\nolimits_{i= 2}^3 \hat{v}_i^2}{\bar{\mathcal{B}}^+ + \hat{v}_5}  D_3^{\hat{v}_6} v_5  \\
		& \quad
		+ H_2 (\mathbf{\hat{v}}, \hat{v}_6) - \frac{1 + |z'|^2}{2 z_1} \frac{\bar{\mathcal{B}}^+ (c^2 - (\bar{u}^+(z_1))^2) - c^2 \sqrt{c^2 - (\bar{u}^+(z_1))^2}}{\gamma \bar{K}^+ \bar{\mathcal{B}}^+} \partial_{z_3} R_4 (\hat{\mathbf{v}}, \hat{v}_6), \\
		\tilde{H}_3 (v_5, \hat{\mathbf{v}}, \hat{v}_6) &= \frac{\hat{v}_3}{\bar{u}^+ (D_0^{\hat{v}_6}) + \hat{v}_1} \tilde{\omega}_1 - \frac{1 + |z'|^2}{2 D_0^{\hat{v}_6} (\bar{u}^+ (D_0^{\hat{v}_6}) + \hat{v}_1)} \frac{c^2 - (\bar{u}^+ (D_0^{\hat{v}_6}) + \hat{v}_1)^2 - \sum\nolimits_{i= 2}^3 \hat{v}_i^2}{\bar{\mathcal{B}}^+ + \hat{v}_5}  D_2^{v_6} v_5 \\
		& \quad
		+ H_3 (\mathbf{\hat{v}}, \hat{v}_6) + \frac{1 + |z'|^2}{2 z_1} \frac{\bar{\mathcal{B}}^+ (c^2 - (\bar{u}^+(z_1))^2) - c^2 \sqrt{c^2 - (\bar{u}^+(z_1))^2}}{\gamma \bar{K}^+ \bar{\mathcal{B}}^+} \partial_{z_2} R_4 (\hat{\mathbf{v}}, \hat{v}_6).
\end{align*}

Moreover, it follows from \eqref{rel-deform-v1} that 
\begin{equation}\label{deform}
	d_1 (z_1) \partial_{z_1} v_1 +  \frac{1 + |z'|^2}{2 z_1} \sum_{j = 2}^3 \partial_{z_j} v_j + \frac{2 v_1}{z_1} - \frac{1}{z_1} \sum_{j = 2}^{3} z_j v_j 
	+ d_2 (z_1) v_1 = \tilde{H}_0 (v_5, \mathbf{\hat{v}}, \hat{v}_6), \quad \text{in } \mathcal{D},
\end{equation}
with $\tilde{H}_0 (v_5, \mathbf{\hat{v}}, \hat{v}_6) = d_0 (D_0^{\hat{v}_6}) v_5 + H (\mathbf{\hat{v}}, \hat{v}_6)$.

Substituting \eqref{v4-exp-1} into the boundary condition \eqref{bdry-4}, we obtain the exit boundary condition
\begin{equation}\label{bdry-5}
	v_1 (r_2, z') + d_3 (r_2) v_1 (r_s, z') = q_2 (z'),
\end{equation}
where 
\begin{equation*}
	\begin{aligned}
		q_2 (z') =&  \frac{\sqrt{c^2 - (\bar{u}^+(r_2))^2}}{\bar{\mathcal{B}}^+ \bar{n}^+(r_2)  \bar{u}^+(r_2)}(E(\mathbf{v}(r_2, z')) -  \epsilon p_{ex} (z')) + \frac{c^2 - (\bar{u}^+ (r_2))^2 }{\bar{\mathcal{B}}^+ \bar{u}^+(r_2)}v_5(r_2, z') \\
		&- d_3 (r_2) R_4  (\hat{\mathbf{v}} (r_s, \boldsymbol{\zeta}(r_2, z')), \hat{v}_6 (\boldsymbol{\zeta}(r_2, z'))).
	\end{aligned}
\end{equation*}
Then, from \eqref{com-p}, \eqref{E}, \eqref{com-Xi}, \eqref{com-v5} and \eqref{com-R4}, we obtain
\begin{equation}\label{com-q2}
	\partial_a q_2 (1, \tau) = 0, \quad \forall \tau \in \mathbb{T}_{2\pi}.
\end{equation}

{\bf (iv) } We now obtain the deformation-curl system \eqref{curl} and \eqref{deform} satisfied by the velocity field, as well as the boundary conditions \eqref{bdry-2}, \eqref{bdry-3}, \eqref{wall-bdry-3} and \eqref{bdry-5}, where $q_1$, $R_1$ and $g_i$ ($i = 2,3$) are evaluated at $(\hat{\mathbf{v}}, \hat{v}_6)$. Due to the linearization, the inhomogeneous terms $(\tilde{H}_1, \tilde{H}_2, \tilde{H}_3)(v_5, \hat{\mathbf{v}}, \hat{v}_6)$ may not be divergence-free, and therefore the solvability condition of the curl system may not hold. To this end, we introduce a new unknown function $\Pi$ to formulate an enlarged deformation-curl system as follows
\begin{equation}\label{deform-curl}
	\begin{cases}
		d_1 (z_1) \partial_{z_1} v_1 +  \frac{1 + |z'|^2}{2 z_1} \sum_{j = 2}^3 \partial_{z_j} v_j + \frac{2 v_1}{z_1} - \frac{1}{z_1} \sum_{j = 2}^{3} z_j v_j + d_2 (z_1) v_1 = \tilde{H}_0 (v_5, \mathbf{\hat{v}}, \hat{v}_6),\\
		\frac{1 + |z'|^2}{2 z_1} (\partial_{z_2} v_3 - \partial_{z_3} v_2) + \frac{1}{z_1} (z_3 v_2 - z_2 v_3) + \partial_{z_1} \Pi= \tilde{H}_1 (\hat{\mathbf{v}}, \hat{v}_6) , \\
		\frac{1 + |z'|^2}{2 z_1} \partial_{z_3}  v_1  - \partial_{z_1} v_3 - \frac{v_3}{z_1} + \frac{1 + |z'|^2}{2 z_1} d_3 (z_1) \partial_{z_3} v_1 (r_s, z') + \frac{1 + |z'|^2}{2 z_1} \partial_{z_2} \Pi
		=  \tilde{H}_2 (v_5, \hat{\mathbf{v}}, \hat{v}_6), \\
		\partial_{z_1} v_2 + \frac{v_2}{z_1} - \frac{1 + |z'|^2}{2 z_1} \partial_{z_2} v_1  - \frac{1 + |z'|^2}{2 z_1} d_3 (z_1) \partial_{z_2} v_1 (r_s, z') + \frac{1 + |z'|^2}{2 z_1} \partial_{z_3} \Pi
		=  \tilde{H}_3 (v_5, \hat{\mathbf{v}}, \hat{v}_6),
	\end{cases}
\end{equation}
with the boundary conditions
\begin{equation}\label{bdry-6}
	\begin{cases}
		(\partial_{z_2}^2 + \partial_{z_3}^2) v_1 (r_s, z') - 2 \alpha_0 a_1 r_s \sum_{j = 2}^{3} \partial_{z_j} \left(\frac{v_j (r_s, z')}{1 + |z'|^2} \right) = q_1({\bf \hat{v}}(r_s,z'),\hat{v}_6(z')) , \quad \forall z' \in E, \\
		v_1 (r_2, z') + d_3 (r_2) v_1 (r_s, z') = q_2 (z'), \quad \forall z' \in E, \\
		\partial_{z_1} \Pi (r_s, z') = \partial_{z_1} \Pi (r_2, z') =0 , \quad \forall z' \in E, \\ 
		(z_2 v_2 + z_3 v_3) (z_1, z') = \Pi (z_1, z') = 0, \quad \forall z_1 \in [r_s , r_2], \quad |z'| = 1, \\
		\sum_{j = 2}^3 z_j \partial_{z_j} v_1 (r_s, z') - \alpha_0 a_1 r_s \sum_{j = 2}^3 z_j v_j (r_s, z') = 0, \quad \text{on } |z'| = 1, 
	\end{cases}
\end{equation}
where the last boundary condition in \eqref{bdry-6} is derived by \eqref{com-N} and \eqref{com-R}.

Moreover, one can obtain the following compatibility condition of $\tilde{H}_i$
\begin{equation}\label{com-H}
	\begin{cases}
		\tilde{H}_1 (\hat{\mathbf{v}}, \hat{v}_6) (z_1, 1, \tau) = 0, \\
		\tilde{H}_{\tau} (v_5, \hat{\mathbf{v}}, \hat{v}_6) (z_1, 1, \tau) = 0, \\
		\partial_a \tilde{H}_a (v_5, \hat{\mathbf{v}}, \hat{v}_6) (z_1, 1, \tau) = 0,
	\end{cases} \quad \forall (z_1, \tau) \in [r_s, r_2] \times \mathbb{T}_{2\pi},
\end{equation}
where $\tilde{H}_a = \cos \tau \tilde{H}_2 + \sin \tau \tilde{H}_3$ and $\tilde{H}_{\tau} = - \sin \tau \tilde{H}_2 + \cos \tau \tilde{H}_3$.

We will prove the existence and uniqueness of solutions to \eqref{deform-curl}-\eqref{bdry-6} in the following steps.

{\bf (a) } Applying the divergence operator to the second, third, and fourth equations of \eqref{deform-curl} yields
\begin{equation}\label{Pi}
	\begin{cases}
		\partial_{z_1}^2 \Pi + \frac{1+|z'|^2}{4 z_1^2}\sum_{j=2}^3 \partial_{z_j}((1+|z'|^2) \partial_{z_j} \Pi) + \frac{2}{z_1} \partial_{z_1} \Pi-\frac{1+|z'|^2}{2 z_1^2}\sum_{j=2}^3 z_j \partial_{z_j} \Pi\\
		\quad\quad\quad= \partial_{z_1} \tilde{H}_1 + \frac{1+|z'|^2}{2 z_1}\sum_{j=2}^3 \partial_{z_j} \tilde{H}_j +\frac{2 \tilde{H}_1}{z_1}- \frac{1}{z_1}\sum_{j=2}^3 z_j \tilde{H}_j,\ \text{in }\mathcal{D},\\
		\partial_{z_1} \Pi(r_s,z')= \partial_{z_1} \Pi(r_2,z')=0,\ \forall z'\in E,\\
		\displaystyle \Pi(z_1,z')=0,\ \ \forall (z_1,z')\in \Gamma_0.
	\end{cases}
\end{equation}
It follows from \cite[Lemma 3.2]{Weng25} that, assuming $\tilde{H}_i \in C^{1, \alpha} (\bar{\mathcal{D}})\, (i = 1, 2,3)$, there exists a unique solution $\Pi \in C^{2, \alpha} (\bar{\mathcal{D}})$ to the problem \eqref{Pi}, which satisfies the estimate
\begin{equation}\label{Pi-est}
	\|\Pi\|_{C^{2,\alpha}(\overline{\mathcal{D}})}\leq C_*\sum_{i=1}^3\|\tilde{H}_i\|_{C^{1,\alpha}(\overline{\mathcal{D}})}\leq C_*(\epsilon\|(\hat{{\bf v}}, \hat{v}_6)\|_{\Xi}+\|(\hat{{\bf v}}, \hat{v}_6)\|_{\Xi}^2).
\end{equation} 

Additionally, expressing the boundary value problem \eqref{Pi} in the coordinate system $(z_1, a, \tau)$ together with the Dirichlet boundary conditions and the compatibility condition \eqref{com-H} yileds that 
\begin{equation}\label{com-Pi}
	\partial_a^2 \Pi (z_1,1,\tau)+ \partial_a \Pi (z_1,1,\tau)=0,\, \forall (z_1,\tau)\in [r_s,r_2]\times \mathbb{T}_{2\pi}.
\end{equation}

{\bf (b) } We then solve the divergence-curl system with normal boundary conditions 
\begin{equation}\label{div-curl}
	\begin{cases}
		\partial_{z_1} \dot{v}_1+\frac{1+|z'|^2}{2 z_1}\sum_{j=2}^3 \partial_{z_j} \dot{v}_j +\frac{2\dot{v}_1}{z_1}-\frac{1}{z_1}\sum_{j=2}^3 z_j \dot{v}_j=0,\ \text{in }\mathcal{D},\\
		\frac{1+|z'|^2}{2z_1}(\partial_{z_2} \dot{v}_3- \partial_{z_3} \dot{v}_2)+ \frac{1}{z_1}(z_3 \dot{v}_2-z_2 \dot{v}_3)=\tilde{H}_1- \partial_{z_1}\Pi \eqqcolon \hat{H}_1,\ \text{in }\mathcal{D},\\
		\frac{1+|z'|^2}{2 z_1} \partial_{z_3} \dot{v}_1- \partial_{z_1} \dot{v}_3-\frac{\dot{v}_3}{z_1}= \tilde{H}_2-\frac{1+|z'|^2}{2z_1} \partial_{z_2}\Pi \eqqcolon \hat{H}_2,\ \text{in }\mathcal{D},\\
		\partial_{z_1} \dot{v}_2+\frac{\dot{v}_2}{z_1}-\frac{1+|z'|^2}{2 z_1} \partial_{z_2} \dot{v}_1= \tilde{H}_3-\frac{1+|z'|^2}{2z_1} \partial_{z_3}\Pi \eqqcolon \hat{H}_3,\ \text{in }\mathcal{D},\\
		\dot{v}_1(r_s,z')= \dot{v}_1(r_2,z')=0,\ \forall z'\in E,\\
		\sum_{j=2}^3 z_j \dot{v}_j (z_1, z')=0,\ \forall (z_1, z')\in \Gamma_0.
	\end{cases}
\end{equation}

It follows from \eqref{com-H} and \eqref{com-Pi} that on $\Gamma_0$ there holds
\begin{equation}\label{com-hat-H}\begin{cases}
		\hat{H}_1(v_5, \hat{{\bf v}},\hat{v}_6)(z_1,1,\tau)=0,\ \ \ \\
		\hat{H}_{\tau}(v_5,\hat{{\bf v}},\hat{v}_6)(z_1,1,\tau)=(\tilde{H}_{\tau}-\frac{1}{z_1}\partial_{\tau}\Pi)(z_1,1,\tau)=0,\ \ \\
		\partial_a \hat{H}_a(v_5,\hat{{\bf v}},\hat{v}_6) (z_1,1,\tau)=(\partial_a \tilde{H}_a-\frac{1}{z_1}(\partial_a^2 \Pi+\partial_a \Pi))(z_1,1,\tau)=0,
	\end{cases}
\end{equation}
where $\hat{H}_a=\cos\tau \hat{H}_2 +\sin\tau \hat{H}_3$ and $\hat{H}_{\tau}=-\sin\tau \hat{H}_2 +\cos\tau \hat{H}_3$.

Since $\Pi$ satisfies the equation in \eqref{Pi}, there holds
\begin{equation}\label{div-hat-H}
	\partial_{z_1} \hat{H}_1+ \frac{1+|z'|^2}{2z_1}\sum_{j=2}^3 \partial_{z_j}\hat{H}_j+\frac{2}{z_1} \hat{H}_1-\frac{1}{z_1}\sum_{j=2}^3 z_j \hat{H}_j\equiv 0,\ \text{in }\mathcal{D}.
\end{equation}
By \cite[Lemma 3.3]{Weng25}, assuming that $(\hat{H}_1, \hat{H}_2, \hat{H}_3) \in (C^{1, \alpha} (\bar{\mathcal{D}}))^3$ satisfies \eqref{com-hat-H} and \eqref{div-hat-H}, the problem \eqref{div-curl} admits a unique solution $(\dot{v}_1, \dot{v}_2, \dot{v}_3) \in (C^{2, \alpha} (\bar{\mathcal{D}}))^3$, which satisfies the estimate
\begin{equation}\label{est-dot-v}
	\sum_{i = 1}^3 \| \dot{v}_i \|_{C^{2, \alpha} (\bar{\mathcal{D}})} \leq C_* \sum_{i = 1}^{3} \| \hat{H}_i \|_{C^{1, \alpha} (\bar{\mathcal{D}})},
\end{equation}
and the compatibility conditions
\begin{equation}\label{com-dot-v}
	\begin{cases}
		(\dot{v}_a,\partial_a \dot{v}_1,\partial_a \dot{v}_{\tau})(z_1,1,\tau)=0,\\
		(\partial_a^2\dot{v}_a+\partial_a \dot{v}_a)(z_1,1,\tau)=0,
	\end{cases} \forall (z_1,\tau)\in [r_s,r_2]\times \mathbb{T}_{2\pi},
\end{equation}
where $\dot{v}_a=\cos\tau \ \dot{v}_2+ \sin \tau \ \dot{v}_3$ and $\dot{v}_{\tau}=-\sin\tau \ \dot{v}_2+ \cos\tau \ \dot{v}_3$.

Let $(v_1, v_2, v_3) $ be the solution to \eqref{deform-curl}, and set $\tilde{v}_i = v_i - \dot{v}_i$ $(i = 1, 2, 3)$. Then $(\tilde{v}_1, \tilde{v}_2, \tilde{v}_3)$ satisfies
\begin{equation}\label{tilde-v}
	\begin{cases}
		d_1(z_1)\partial_{z_1} \tilde{v}_1+\frac{1+|z'|^2}{2z_1}\sum_{j=2}^3 \partial_{z_j} \tilde{v}_j +\frac{2 \tilde{v}_1}{z_1}-\frac{1}{z_1}\sum_{j=2}^3 z_j \tilde{v}_j + d_2(z_1) \tilde{v}_1 = \tilde{H}_4(v_5, \hat{{\bf v}},\hat{v}_6),\ \text{in }\mathcal{D},\\
		\frac{1+|z'|^2}{2 z_1}(\partial_{z_2} \tilde{v}_3-\partial_{z_3} \tilde{v}_2)+ \frac{1}{z_1}(z_3 \tilde{v}_2-z_2 \tilde{v}_3)=0,\  \text{in }\mathcal{D},\\
		\frac{1+|z'|^2}{2z_1}\partial_{z_3} (\tilde{v}_1+d_3(z_1) \tilde{v}_1(r_s,z'))-\partial_{z_1} \tilde{v}_3-\frac{\tilde{v}_3}{z_1}= 0,\ \text{in }\mathcal{D},\\
		\partial_{z_1} \tilde{v}_2-\frac{1+|z'|^2}{2z_1}\partial_{z_2} (\tilde{v}_1+d_3(z_1) \tilde{v}_1(r_s,z'))+\frac{\tilde{v}_2}{z_1} = 0,\ \text{in }\mathcal{D},\\
		\sum_{j=2}^3\partial_{z_j}^2 \tilde{v}_1(r_s,z')- \alpha_0 a_1\partial_{z_j}\big(\frac{2r_s \tilde{v}_j}{1+|z'|^2}\big)(r_s,z')=q_3(\hat{{\bf v}}(r_s,z'),\hat{v}_6(z')),\ \forall z'\in E,\\
		\tilde{v}_1(r_2,z')+ d_3(r_2) \tilde{v}_1(r_s,z')=q_4(z'),\ \forall z'\in E,\\
		z_2 \tilde{v}_2(z_1,z')+ z_3 \tilde{v}_3(z_1,z')=0,\ \forall (z_1,z')\in \Gamma_0,\\
		\sum_{j=2}^3 z_j \partial_{z_j} \tilde{v}_1(r_s,z')-\alpha_0 a_1 r_s\sum_{j=2}^3 z_j \tilde{v}_j(r_s,z')=0, \text{ on } |z'|=1,
	\end{cases}
\end{equation}
where
\begin{align*}
	&\tilde{H}_4(v_5, \hat{{\bf v}},\hat{v}_6)= \tilde{H}_0 (v_5, \hat{{\bf v}},\hat{v}_6)+\bar{\mathcal{M}}^2(z_1)\partial_{z_1}\dot{v}_1-d_2(z_1)\dot{v}_1,\\
	&q_3(\hat{{\bf v}}(r_s,z'),\hat{v}_6(z'))=q_1(\hat{{\bf v}}(r_s,z'),\hat{v}_6(z'))+2r_s \alpha_0 a_1\sum_{j=2}^3\partial_{z_j}\big(\frac{\dot{v}_j}{1+|z'|^2}\big)(r_s,z').
\end{align*}

Rewriting the second, the third and the fourth equation in \eqref{tilde-v} as
\begin{equation*}\begin{cases}
		\partial_{z_2}\big(\frac{2 z_1 \tilde{v}_3}{1+|z'|^2}\big)-\partial_{z_3}\big(\frac{2 z_1 \tilde{v}_2}{1+|z'|^2}\big)=0,  \\
		\partial_{z_3} (\tilde{v}_1+d_3(z_1) \tilde{v}_1(r_s,z'))-\partial_{z_1} \big( \frac{z_1 \tilde{v}_3}{1 + |z'|^2} \big) = 0,  \\
		\partial_{z_1} \big( \frac{z_1 \tilde{v}_2}{1 + |z'|^2} \big) - \partial_{z_2} (\tilde{v}_1+d_3(z_1) \tilde{v}_1(r_s,z')) = 0, 
	\end{cases} \text{in } \mathcal{D}.
\end{equation*}
Hence a potential function $\psi$ can be introduced such that
\begin{equation*}
	\partial_{z_1} \psi = \tilde{v}_1 (z_1, z') + d_3 (z_1) \tilde{v}_1 (r_s, z'), \quad
	\partial_{z_j} \psi = \frac{2 z_1 \tilde{v}_j}{1 + |z'|^2}, 
	\, 
	j = 2, 3.
\end{equation*}
Consequently, 
\begin{equation*}\begin{cases}
		\tilde{v}_1 (r_s, z') = \frac{1}{1 + d_3 (r_s)} \partial_{z_1} \psi (r_s, z'), \\ 
		\tilde{v}_1 (z_1, z') = \partial_{z_1} \psi (z_1, z') - \frac{d_3 (z_1)}{1 + d_3 (r_s)} \partial_{z_1} \psi (r_s, z'), \\ 
		\tilde{v}_j (z_1, z') = \frac{1 + |z'|^2}{2 z_1} \partial_{z_j} \psi (z_1, z'), \, j= 2, 3,
	\end{cases}
\end{equation*}
and thus the problem \eqref{tilde-v} can be reformulated as
\begin{equation}\label{psi}
	\begin{cases}
		d_1(z_1) \partial_{z_1}^2 \psi
		+ \frac{(1+|z'|^2)^2}{4 z_1^2} \sum_{j=2}^3 \partial_{z_j}^2  \psi + \frac{2}{z_1} \partial_{z_1} \psi
		+ d_2(z_1) \partial_{z_1} \psi
		- \frac{ d_4(z_1)}{a_3} \partial_{z_1} \psi(r_s, z')
		= \tilde{H}_4, \quad \text{in } \mathcal{D}, \\
		\sum_{j=2}^3 \partial_{z_j}^2 \left( \partial_{z_1} \psi(r_s, z')
		- a_4  \psi(r_s, z') \right)
		= a_3 q_3(z'), \quad \forall z' \in E, \\
		\partial_{z_1} \psi(r_2, z')  = q_4(z'), \quad \forall z' \in E, \\
		z_2 \partial_{z_2} \psi + z_3 \partial_{z_3} \psi = 0, \quad \forall (z_1, z') \in \Gamma_0, \\
		\sum_{j=2}^3 z_j \partial_{z_j} \left( \partial_{z_1} \psi(r_s, z')
		- a_4  \psi(r_s, z') \right) = 0, \quad \text{on } z_1 = r_s,\ |z'| = 1,
	\end{cases}
\end{equation}
where
\begin{align*}
	a_3 = & 1 + d_3(r_s) > 0, \quad a_4 = \alpha_0 a_1 a_3 > 0, \\
	d_4 (z_1) = & d_1(z_1) d_3'(z_1) + \left( \frac{2}{z_1} + d_2(z_1) \right) d_3(z_1)  \\
	= &  \frac{a_2}{a_1} \frac{1}{\gamma \bar{K}^+ \bar{\mathcal{B}}^+} \frac{2}{z_1 \bar{u}^+ (z_1)  } \left( \bar{\mathcal{B}}^+ (c^2 + (\bar{u}^+ (z_1))^2) - \frac{c^4}{\sqrt{c^2 - (\bar{u}^+ (z_1))^2}} \right) \\
	& 
	+ \frac{2}{z_1} \left( 1 + \frac{c^2}{1 - \bar{\mathcal{M}}^2 (z_1)} \left( \frac{2 \bar{\mathcal{M}}^2 (z_1)}{c^2 - (\bar{u}^+ (z_1))^2} + \frac{(\gamma - 1) c^4 (\bar{u}^+ (z_1))^4}{\bar{\mathcal{B}}^+ \bar{c}_s^4 (z_1) (c^2 - (\bar{u}^+ (z_1))^2)^{5/2} } \right) \right) \\
	& \quad
	\times \frac{a_2}{a_1 \gamma \bar{K}^+ \bar{u}^+ (z_1)} \frac{\sqrt{c^2 - (\bar{u}^+ (z_1))^2}}{\bar{\mathcal{B}}^+} (\bar{\mathcal{B}}^+{\sqrt{c^2 - (\bar{u}^+ (z_1))^2}} - c^2) \\
	= & 
	\frac{2 a_2}{\gamma \bar{K}^+ \bar{\mathcal{B}}^+ a_1} \cdot \frac{1}{z_1 \bar{u}^+ (z_1) } \left( \bar{\mathcal{B}}^+  (\bar{u}^+ (z_1))^2 + \frac{c^2 (\bar{\mathcal{B}}^+{\sqrt{c^2 - (\bar{u}^+ (z_1))^2}} - c^2)}{\sqrt{c^2 - (\bar{u}^+ (z_1))^2}}  \right) \\
	& 
	+ \frac{2 a_2}{\gamma \bar{K}^+ \bar{\mathcal{B}}^+ a_1} \cdot \frac{\sqrt{c^2 - (\bar{u}^+ (z_1))^2}}{z_1 \bar{u}^+ (z_1) } \cdot
	\left(\bar{\mathcal{B}}^+{\sqrt{c^2 - (\bar{u}^+ (z_1))^2}} - c^2 \right) \\
	& \quad
	\times \left( 1 + \frac{c^2}{1 - \bar{\mathcal{M}}^2 (z_1)} \left( \frac{2 \bar{\mathcal{M}}^2 (z_1)}{c^2 - (\bar{u}^+ (z_1))^2} + \frac{(\gamma - 1) c^4 (\bar{u}^+ (z_1))^4}{\bar{\mathcal{B}}^+ \bar{c}_s^4 (z_1) (c^2 - (\bar{u}^+ (z_1))^2)^{5/2} } \right) \right) > 0.
\end{align*}

It is crucial to derive the oblique boundary condition for the potential $\psi$ on the boundary $\{(r_s,z')\mid z'\in E\}$ by solving the Poisson equation with homogeneous Neumann boundary condition from the first and fourth boundary conditions in \eqref{psi}.
\begin{lemma}\label{oblique}({\bf The oblique boundary condition on the shock front.})
	On the shock front $\{(r_s, z')\mid z'\in E\}$, there exists a unique $C^{2,\alpha}(\overline{E})$ function $q_s(y')$ such that
	\begin{equation*}
		\partial_{z_1}\psi(r_s,z')-a_4 \psi(r_s,z')= q_s(z'),
	\end{equation*}
	where $q_s(z')$ satisfies the Poisson equation with Neumann boundary conditions
	\begin{equation}\label{den38}\begin{cases}
			(\partial_{z_2}^2+\partial_{z_3}^2) q_s(z')=a_3 q_3(\hat{{\bf v}}(r_s,z'),\hat{v}_6(z')),\ \ &\text{in } E,\\
			(z_2\partial_{z_2}+z_3\partial_{z_3}) q_s(z')=0,\ &\forall |z'|=1,\\
			\iint_{E} q_s(z') dz'=0,
	\end{cases}\end{equation}
	and the following estimate holds
	\begin{equation}\label{den39}
		\|q_s\|_{C^{2,\alpha}(\overline{E})}\leq C_*\|q_3(\hat{{\bf v}}(r_s,z'),\hat{v}_6(z'))\|_{C^{\alpha}(\overline{E})}.
	\end{equation}
\end{lemma}

Then the problem \eqref{psi} can be rewritten as 
\begin{equation}\label{psi-2}
	\begin{cases}
		d_1(z_1) \partial_{z_1}^2 \psi
		+ \frac{(1+|z'|^2)^2}{4 z_1^2} \sum_{j=2}^3 \partial_{z_j}^2  \psi + \frac{2}{z_1} \partial_{z_1} \psi
		+ d_2(z_1) \partial_{z_1} \psi \\
		\qquad
		- \alpha_0 a_1 d_4(z_1) \partial_{z_1} \psi(r_s, z')
		= \tilde{H}_5, \quad \text{in } \mathcal{D}, \\
		\partial_{z_1} \psi(r_s, z')
		- a_4  \psi(r_s, z') 
		= q_s (z'), \quad \forall z' \in E, \\
		\partial_{z_1} \psi(r_2, z')  = q_4(z'), \quad \forall z' \in E, \\
		z_2 \partial_{z_2} \psi + z_3 \partial_{z_3} \psi = 0, \quad \forall (z_1, z') \in \Gamma_0,
	\end{cases}
\end{equation}
where $\tilde{H}_5 (z) = \tilde{H}_4 (z) + \frac{d_4 (z_1)}{a_3} q_s (z')$. Furthermore, direct calculations yield that
\begin{equation}\label{com-H5}\begin{aligned}
		& \partial_a \tilde{H}_5 (z_1, 1, \tau) = 0, \quad \forall (z_1, \tau) \in [r_s, r_2] \times \mathbb{T}_{2\pi}, \\
		& \partial_a q_i (1, \tau) = 0, \quad \forall \tau \in \mathbb{T}_{2\pi}, \, i = s, 4.
\end{aligned}\end{equation}
The  problem \eqref{psi-2} can be solved by analogy with the method in \cite[Proposition 3.8]{Weng25}.

\begin{proposition}\label{solvability}
	{\it Suppose that $(q_s,q_4)\in (C^{2,\alpha}(\overline{E}))^2$ and $\tilde{H}_5\in C^{1,\alpha}(\overline{\mathcal{D}})$ satisfy the compatibility conditions \eqref{com-H5}. Then there exists a unique solution $\psi\in C^{3,\alpha}(\overline{\mathcal{D}})$ to the problem \eqref{psi-2} with the estimate
		\begin{equation}\label{est-psi}
			\|\psi\|_{C^{3,\alpha}(\overline{\mathcal{D}})}\leq C_*(\|\tilde{H}_5\|_{C^{1,\alpha}(\overline{\mathcal{D}})}+\sum_{j=s,4}\|q_j\|_{C^{2,\alpha}(\overline{E})}),
		\end{equation}
		where the constant $C_*$ depends only on the coefficients $d_1,d_4,a_3,a_4$ and thus depends only the background solution.
}\end{proposition}

Thus $\tilde{v}_1(z)=\partial_{z_1}\psi(z)-\frac{d_3 (z_1)}{a_3} \partial_{z_1}\psi(r_s,z'), \tilde{v}_j(z)=\frac{1+|z'|^2}{2z_1}\partial_{z_j}\psi(z), j=2,3$ would solve the problem \eqref{tilde-v}. Furthermore, one can derive the following compatibility conditions
\begin{equation*}\begin{cases}
		\tilde{v}_a(z_1,1,\tau)=\partial_a \tilde{v}_1(z_1,1,\tau)=\partial_a \tilde{v}_{\tau}(z_1,1,\tau)=0,\\
		(\partial_a^2 \tilde{v}_a+\partial_a \tilde{v}_a)(z_1,1,\tau)=0,
\end{cases} \quad \forall (z_1,\tau)\in [r_s,r_2]\times \mathbb{T}_{2\pi},
\end{equation*}
where $\tilde{v}_a=\cos\tau \,\tilde{v}_2+ \sin \tau \, \tilde{v}_3$ and $\tilde{v}_{\tau}=-\sin\tau \, \tilde{v}_2+ \cos\tau \, \tilde{v}_3$.

Then
\begin{align*}
	&v_1(z_1,z')=\dot{v}_1(z)+\tilde{v}_1(z)=\dot{v}_1(z)+\partial_{z_1}\psi(z)-\frac{d_3 (z_1)}{a_3} \partial_{z_1}\psi(r_s,z'),\\
	&v_j(z_1,z')=\dot{v}_j(z)+\tilde{v}_j(z)=\dot{v}_j(z)+\frac{1+|z'|^2}{2z_1}\partial_{z_j}\psi(z),\ j=2,3,
\end{align*}
will solve the problem \eqref{deform-curl} and satisfy the estimate
\begin{equation}\label{est-v}\begin{aligned}
		\sum_{j=1}^3\|v_j\|_{C^{2,\alpha}(\overline{\mathcal{D}})}&\leq C_*(\sum_{j=1}^3\|\dot{v}_j\|_{C^{2,\alpha}(\overline{\mathcal{D}})}+\|\nabla \psi\|_{C^{2,\alpha}(\overline{\mathcal{D}})}+ \|\partial_{z_1}\psi(r_s,z')\|_{C^{2,\alpha}(\overline{E})})\\
		&\leq C_*(\epsilon +C_*(\epsilon\|(\hat{{\bf v}}, \hat{v}_6)\|_{\Xi}+\|(\hat{{\bf v}}, \hat{v}_6)\|_{\Xi}^2)\leq C_*(\epsilon+\epsilon \sigma_0 +\sigma_0^2),
	\end{aligned}
\end{equation}
and the compatibility conditions
\begin{equation*}\begin{cases}
		v_a(z_1,1,\tau)=\partial_a v_1(z_1,1,\tau)=\partial_a v_{\tau}(z_1,1,\tau)=0,\\
		(\partial_a^2 v_a+\partial_a v_a)(z_1,1,\tau)=0,
\end{cases}\quad\forall (z_1,\tau)\in [r_s,r_2]\times \mathbb{T}_{2\pi}, \end{equation*}
where $v_a=\cos\tau \, v_2+ \sin \tau \, v_3$ and $v_{\tau}=-\sin\tau \, v_2+ \cos\tau \, v_3$.

{\bf (v) } Once $v_1, v_2, v_3$ are obtained, the function $v_4$ is uniquely determined by \eqref{v4-exp-1}:
\begin{equation*}
	v_4(z_1, z') = \frac{a_2}{a_1} v_1(r_s, z') + R_4(\hat{\mathbf{v}}(r_s, \zeta_2(z), \zeta_3(z)), \hat{v}_6(\zeta_2(z), \zeta_3(z))).
\end{equation*}
There also holds the estimate
\begin{equation}\label{est-v4}
	\|v_4\|_{C^{2,\alpha}(\overline{\mathcal{D}})} \leq C_* \|v_1(r_s, \cdot)\|_{C^{2,\alpha}(\overline{E})} + C_* \bigl( \epsilon \|(\hat{\mathbf{v}}, \hat{v}_6)\|_{\Xi} + \|(\hat{\mathbf{v}}, \hat{v}_6)\|_{\Xi}^2 \bigr) 
	\leq C_* (\epsilon \sigma_0 + \sigma_0^2),
\end{equation}
and the compatibility condition
\begin{equation*}
	\partial_a v_4(z_1, 1, \tau) = \frac{a_2}{a_1} \partial_a v_1(r_s, 1, \tau) = 0, \quad \forall (z_1, \tau) \in \Gamma_0.
\end{equation*}

Finally, the shock front is uniquely determined by
\begin{equation}\label{shock50}
	v_6(z') = \frac{1}{a_1} v_1(r_s, z') - \frac{1}{a_1} R_1(\hat{\mathbf{v}}(r_s, z'), \hat{v}_6(z')),
\end{equation}
which implies that $v_6 \in C^{2,\alpha}(\overline{E})$. And the compatibility condition also holds
\begin{equation*}
	\partial_a v_6(1, \tau) = \frac{1}{a_1} \partial_a v_1(r_s, 1, \tau) - \frac{1}{a_1} \partial_a \{R_1(\hat{\mathbf{v}}(r_s, \cdot), \hat{v}_6)\}(1, \tau) = 0, \quad \forall \tau \in \mathbb{T}_{2\pi}.
\end{equation*}
We now improve the regularity of $v_6$ to be $C^{3,\alpha}(\overline{E})$. To this end, define
\begin{equation*}
	\begin{aligned}
		F_i(z')  \coloneqq & \partial_{z_i} v_1(r_s, z') - \frac{2 \alpha_0 r_s v_i(r_s, z')}{1 + |z'|^2} - \frac{2 a_1 g_i(\hat{\mathbf{v}}(r_s, z'), \hat{v}_6(z'))}{1 + |z'|^2} \\
		& - \partial_{z_i} \{R_1(\hat{\mathbf{v}}(r_s, z'), \hat{v}_6(z'))\}, \quad i = 2, 3.
	\end{aligned}
\end{equation*}
Then it follows from the first boundary condition in \eqref{bdry-6} and the boundary data in \eqref{omega1} that
\begin{equation*}
	\begin{cases}
		\partial_{z_2} F_2 + \partial_{z_3} F_3 = 0, & \text{in } E, \\
		\partial_{z_2} F_3 - \partial_{z_3} F_2 = 0, & \text{in } E, \\
		z_2 F_2 + z_3 F_3 = 0, & \text{on } \partial E.
	\end{cases}
\end{equation*}
Thus by Lemma \ref{equi0}, $F_2 = F_3 \equiv 0$ in $E$. Using the equation \eqref{shock50}, there holds
\begin{equation*}
	\begin{cases}
		\partial_{z_2} v_6(z') = \dfrac{2 \alpha_0 r_s v_2(r_s, z')}{1 + |z'|^2} + \dfrac{2 g_2(\hat{\mathbf{v}}(r_s, z'), \hat{v}_6(z'))}{1 + |z'|^2}, \\
		\partial_{z_3} v_6(z') = \dfrac{2 \alpha_0 r_s v_3(r_s, z')}{1 + |z'|^2} + \dfrac{2 g_3(\hat{\mathbf{v}}(r_s, z'), \hat{v}_6(z'))}{1 + |z'|^2}, 
	\end{cases} \quad  \text{in } E.
\end{equation*}
Therefore $v_6 \in C^{3,\alpha}(\overline{E})$ with the estimate
\begin{equation}\label{est-v6}
	\begin{aligned}
		\|v_6\|_{C^{3,\alpha}(\overline{E})} &\leq C_* \|v_1(r_s, \cdot)\|_{C^{2,\alpha}(\overline{E})} + C_* \|R_1(\hat{\mathbf{v}}(r_s, z'), \hat{v}_6(z'))\|_{C^{2,\alpha}(\overline{E})}  \\
		&\quad + C_* \sum_{j=2}^3 \bigl( \|v_j(r_s, \cdot)\|_{C^{2,\alpha}(\overline{E})} + \|g_j(\hat{\mathbf{v}}(r_s, z'), \hat{v}_6(z'))\|_{C^{2,\alpha}(\overline{E})} \bigr) \\
		&\leq C_* \bigl( \epsilon + \epsilon \|(\hat{\mathbf{v}}, \hat{v}_6)\|_{\Xi} + \|(\hat{\mathbf{v}}, \hat{v}_6)\|_{\Xi}^2 \bigr)
		\leq C_* (\epsilon + \epsilon \sigma_0 + \sigma_0^2).
	\end{aligned}
\end{equation}

Combining the estimates \eqref{est-v5},  \eqref{est-v}, \eqref{est-v4} and \eqref{est-v6}, one concludes that
\begin{equation*}
	\|(\mathbf{v}, v_6)\|_{\Xi} = \sum_{j=1}^5 \|v_j\|_{C^{2,\alpha}(\overline{\mathcal{D}})} + \|v_6\|_{C^{3,\alpha}(\overline{E})} \leq C_* (\epsilon + \epsilon \sigma_0 + \sigma_0^2) \leq C_* (\epsilon + \sigma_0^2).
\end{equation*}
Choose $\sigma_0 = \sqrt{\epsilon}$ and let $\epsilon < \epsilon_0 = \frac{1}{4 C_*^2}$. Then $\|(\mathbf{v}, v_6)\|_{\Xi} \leq 2 C_* \epsilon \leq \sigma_0$, thus $(\mathbf{v}, v_6) \in \Xi$. We now can define the operator $\mathcal{P} : (\hat{\mathbf{v}}, \hat{v}_6) \mapsto (\mathbf{v}, v_6)$ which maps $\Xi$ to itself.

{\bf (vi) } It remains to show that $\mathcal{P}$ is a contraction in the weak norm
\begin{equation*}
	\|(\mathbf{v}, v_6)\|_w \coloneqq \sum_{j=1}^5 \|v_j\|_{C^{1,\alpha}(\overline{\mathcal{D}})} + \|v_6\|_{C^{2,\alpha}(\overline{E})}.
\end{equation*}
This can be done by taking the difference for two solutions, we omit the details. Since the mapping $\mathcal{P}$ is a contraction operator in the weak norm $\|\cdot\|_w$, there exists a unique point $(\mathbf{v}, v_6) \in \Xi$ such that $\mathcal{P}(\mathbf{v}, v_6) = (\mathbf{v}, v_6)$. It remains to prove that the auxiliary function $\Pi$ associated with the fixed point $(\mathbf{v}, v_6)$ in solving the problem \eqref{deform-curl}, is automatically $\Pi \equiv 0$ in $\mathcal{D}$. Thanks to the definitions of $\tilde{H}_j(\mathbf{v}, v_6)$ for $j = 1, 2, 3$, one may infer from \eqref{deform-curl} that
\begin{equation}\label{pi0}
	\begin{cases}
		-\partial_{z_1} \Pi = \dfrac{1 + |z'|^2}{2 D_0^{v_6}} (D_2^{v_6} v_3 - D_3^{v_6} v_2) + \dfrac{1}{D_0^{v_6}} (z_3 v_2 - z_2 v_3) - \tilde{\omega}_1, \\[12pt]
		-\dfrac{1 + |z'|^2}{2 z_1} \partial_{z_2} \Pi = \dfrac{1 + |z'|^2}{2 D_0^{v_6}} D_3^{v_6} v_1 - D_1^{v_6} v_3 - \dfrac{v_3}{D_0^{v_6}} - \dfrac{v_2 \tilde{\omega}_1}{\bar{u}^+(D_0^{v_6}) + v_1} \\
		\quad - \dfrac{(1 + |z'|^2) D_3^{v_6} v_5}{2 D_0^{v_6} (\bar{u}^+(D_0^{v_6}) + v_1)} + \dfrac{\bar{\mathcal{B}}^+ + v_5 - \frac{1}{2}(\bar{u}^+(D_0^{v_6}) + v_1)^2 - \frac{1}{2}(v_2^2 + v_3^2)}{\gamma(\bar{K}^+ + v_4)} \dfrac{(1 + |z'|^2) D_3^{v_6} v_4}{2 D_0^{v_6} (\bar{u}^+(D_0^{v_6}) + v_1)}, \\[12pt]
		-\dfrac{1 + |z'|^2}{2 z_1} \partial_{z_3} \Pi = D_1^{v_6} v_2 + \dfrac{v_2}{D_0^{v_6}} - \dfrac{1 + |z'|^2}{2 D_0^{v_6}} D_2^{v_6} v_1 - \dfrac{v_3 \tilde{\omega}_1}{\bar{u}^+(D_0^{v_6}) + v_1} \\
		\quad + \dfrac{(1 + |z'|^2) D_2^{v_6} v_5}{2 D_0^{v_6} (\bar{u}^+(D_0^{v_6}) + v_1)} - \dfrac{\bar{\mathcal{B}}^+ + v_5 - \frac{1}{2}(\bar{u}^+(D_0^{v_6}) + v_1)^2 - \frac{1}{2}(v_2^2 + v_3^2)}{\gamma(\bar{K}^+ + v_4)} \dfrac{(1 + |z'|^2) D_2^{v_6} v_4}{2 D_0^{v_6} (\bar{u}^+(D_0^{v_6}) + v_1)}.
	\end{cases}
\end{equation}
Since the vorticity $\tilde{\omega}_1$ satisfies the equation \eqref{omega-1} and the following commutator relations hold
\begin{equation*}
	D_1^{v_6} D_2^{v_6} = D_2^{v_6} D_1^{v_6}, \quad D_2^{v_6} D_3^{v_6} = D_3^{v_6} D_2^{v_6}, \quad D_1^{v_6} D_3^{v_6} = D_3^{v_6} D_1^{v_6},
\end{equation*}
one can conclude from \eqref{pi0} that in $\mathcal{D}$
\begin{equation*}
	D_1^{v_6}(\partial_{z_1} \Pi) + \frac{1 + |z'|^2}{4 z_1 D_0^{v_6}} \sum_{j=2}^3 D_j^{v_6}((1 + |z'|^2) \partial_{z_j} \Pi) + \frac{2 \partial_{z_1} \Pi}{D_0^{v_6}} - \frac{1 + |z'|^2}{2 z_1 D_0^{v_6}} \sum_{j=2}^3 z_j \partial_{z_j} \Pi = 0.
\end{equation*}
Since $\|v_6\|_{C^{3,\alpha}(\overline{E})} \leq \sigma_0$, where $\sigma_0$ is sufficiently small, thus $\Pi$ satisfies a second order uniformly elliptic equation without zeroth order term. Thanks to the homogeneous mixed boundary conditions for $\Pi$ on $\partial \mathcal{D}$ in \eqref{deform-curl}, it follows directly from the maximum principle that $\Pi \equiv 0$ in $\mathcal{D}$. Thus $(\mathbf{v}, v_6)$ is the desired solution. The proof of Theorem \ref{existence} is completed.

\section{Appendix}\label{appendix}
In this section, we will give a proof of Proposition~\ref{1d-existence}.
\begin{proof}
	
	\textbf{Step 1.} Given the incoming supersonic flow $(\bar{u}_0^-, \bar{n}_0^-, \bar{S}_0^-)$, then there exists a unique smooth supersonic solution $(\bar{u}^-(y_1), \bar{n}^-(y_1), \bar{S}_0^-)$ on the interval $[r_1, r_2]$.
	
	In fact, it follows from \eqref{rel-1d-euler} that 
	\begin{equation}\label{rel-const}
		\begin{cases}
			J (\bar{u}^{\pm} , \bar{n}^{\pm}, y_1) \equiv y_1^2 \frac{\bar{n}^{\pm} (y_1) \bar{u}^{\pm} (y_1)}{\sqrt{c^2 - (\bar{u}^{\pm} (y_1))^2}} - J_0 = 0, \\
			\mathcal{B} (\bar{u}^{\pm} , \bar{n}^{\pm}, y_1) \equiv \frac{\bar{p}^{\pm}(y_1) + \bar{\rho}^{\pm} (y_1) c^2}{\bar{n}^{\pm} (y_1) \sqrt{c^2 - (\bar{u}^{\pm} (y_1))^2} } - \mathcal{B}_0 = 0,
		\end{cases}
	\end{equation}
	with $J_0 = r_1^2  \frac{\bar{n}^- (r_1) \bar{u}^- (r_1)}{\sqrt{c^2 - (\bar{u}^- (r_1))^2}}$ and $\mathcal{B}_0 = \frac{\bar{p}^-(r_1) + \bar{\rho}^- (r_1) c^2}{\bar{n}^- (r_1) \sqrt{c^2 - (\bar{u}^- (r_1))^2} }$. Moreover, 
	\begin{equation}\label{rel-sound}
		c_s^2(\bar{n}, \bar{S})  = \frac{\frac{\partial \bar{p}}{\partial \bar{n}} }{\frac{{\partial \bar{\rho}}}{\partial \bar{n}}} = \frac{(\gamma - 1) c^2 (\bar{\mathcal{B}} \sqrt{c^2 - \bar{u}^2 } - c^2)}{\bar{\mathcal{B}} \sqrt{c^2 - \bar{u}^2 }}, 
	\end{equation}
	then 
	\begin{equation}\label{derivative}
		\frac{d}{d y_1} ((\bar{u}^- (y_1))^2 - c_s^2(\bar{n}_0^-(y_1), \bar{S}_0^-)) =  \left( 2 + \frac{(\gamma - 1)c^4}{\bar{\mathcal{B}} (c^2 - (\bar{u}^- (y_1))^2)^{3/2}} \right) \bar{u}^- (y_1) (\bar{u}^-)' (y_1).
	\end{equation}
	In addition, we have 
	\begin{equation}\label{deriv-u}
		\frac{d \bar{u}^-}{d y_1} = - \frac{2 c_s^2 (\bar{n}_0^-, \bar{S}_0^-) \bar{u}^- (c^2 - (\bar{u}^-)^2)}{y_1 c^2 (c_s^2 (\bar{n}_0^-, \bar{S}_0^-) - (\bar{u}^-)^2) }.
	\end{equation}
	Together with \eqref{derivative}, one gets
	\begin{equation}\label{positive}
		(\bar{u}^- (y_1))^2 - c_s^2(\bar{n}_0^-(y_1), \bar{S}_0^-) \geq (\bar{u}^- (r_1))^2 - c_s^2(\bar{n}_0^-(r_1), \bar{S}_0^-) > 0, \quad  y_1 \geq r_1.
	\end{equation}
	
	Since
	\begin{equation*}
		\det \left( \frac{\partial (J, \mathcal{B})}{\partial (\bar{u}^-, \bar{n}^-)} \right) = \frac{y_1^2 (c_s^2 (\bar{n}^-, \bar{S}^-) - (\bar{u}^-)^2)}{\bar{\mathcal{B}} (c^2 - (\bar{u}^-)^2)},
	\end{equation*}
	and
	\begin{equation*}
		\det \left( \frac{\partial (J, \mathcal{B})}{\partial (\bar{u}^-, \bar{n}^-)} \right)\bigg|_{\bar{u}^- (r_1), \bar{n}^-(r_1), r_1} < 0.
	\end{equation*}
	Combining this with the implicit function theorem and \eqref{positive} leads to the existence of a unique supersonic solution $(\bar{u}^-(y_1), \bar{n}^-(y_1), \bar{S}_0^-)$ on $[r_1, r_2]$.
	
	\textbf{Step 2.} Suppose that a shock front is located at $r_s \in (r_1, r_2)$, then the subsonic state $(\bar{u}^+ (r_s), \bar{n}^+(r_s), \bar{S}_0^+)$ can be uniquely determined by the supersonic state $(\bar{u}^-(r_s), \bar{n}^-(r_s), \bar{S}_0^-)$ via the Rankine-Hugoniot conditions and the entropy condition \eqref{RH-jump}.
	
	In fact, from \eqref{RH-jump}, we have
	\begin{equation}\label{RH-const}
		\begin{cases}
			\frac{\bar{n}^+ (r_s) \bar{u}^+ (r_s)}{\sqrt{c^2 - (\bar{u}^+ (r_s))^2}} = \frac{\bar{n}^- (r_s) \bar{u}^- (r_s)}{\sqrt{c^2 - (\bar{u}^- (r_s))^2}} \equiv \bar{J} , \\
			\frac{\bar{p}^+ (r_s) + \bar{\rho}^+ (r_s) c^2}{c^2 - (\bar{u}^+ (r_s))^2} \, (\bar{u}^+ (r_s))^2 + \bar{p}^+ (r_s) = \frac{\bar{p}^- (r_s) + \bar{\rho}^- (r_s) c^2}{c^2 - (\bar{u}^- (r_s))^2} \, (\bar{u}^- (r_s))^2 + \bar{p}^- (r_s) \equiv \bar{\Phi} , \\
			\bar{\mathcal{B}}^+ (r_s) = \bar{\mathcal{B}}^- (r_s) \equiv \bar{\mathcal{B}}.
		\end{cases}
	\end{equation}
	Then $ \bar{p}^+ (r_s) - \bar{p}^- (r_s) = \bar{\mathcal{B}} \bar{J} (\bar{u}^- (r_s) - \bar{u}^+ (r_s))$, and the physical entropy condition implies that $\bar{u}^- (r_s) > \bar{u}^+ (r_s)$.
	
	Let
	\begin{equation*}
		g (u) = \frac{(\gamma - 1)\bar{J} }{\gamma} \frac{\bar{\mathcal{B}} (c^2 - u^2) - c^2 \sqrt{c^2 - u^2}}{u}.
	\end{equation*}
	Then, by \eqref{rho}, \eqref{eq:internal-energy}, \eqref{rel-bernoulli} and \eqref{RH-const}, we obtain $\bar{p}^{\pm} (r_s) = g (\bar{u}^{\pm} (r_s))$. 
	
	Define $\Phi (u) = \bar{\mathcal{B}} \bar{J}  u + g (u) - \bar{\Phi} (r_s)$. Then $\Phi (\bar{u}^+ (r_s)) = \Phi (\bar{u}^- (r_s)) = 0$ and 
	\begin{equation*}
		\Phi'(u) = \bar{\mathcal{B}} \bar{J}  + g' (u) = \frac{\bar{\mathcal{B}} \bar{J}}{\gamma u^2} (u^2 - \bar{c}_s^2).
	\end{equation*}
	Consequently, we can prove the existence of a unique subsonic state $(\bar{u}^+ (r_s), \bar{n}^+(r_s), \bar{S}_0^+)$.

	\textbf{Step 3.} Taking the subsonic state $(\bar{u}^+(r_s), \bar{n}^+(r_s), \bar{S}^+)$ as the initial condition, a similar argument to \textbf{Step 1} yields the existence of a unique smooth subsonic solution $(\bar{u}^+(y_1), \bar{n}^+(y_1), \bar{S}^+)$ on $[r_s, r_2]$.
	
	\textbf{Step 4.} We proceed to establish the monotonic relationship between the shock position $r_s$ and the exit pressure $p_e$.
	
	Form \eqref{rel-const}, one has
	\begin{equation}\label{rel-deriv}
		\begin{cases}
			\frac{\mathrm{d}}{\mathrm{d} p_e} \left( \frac{r_s^2 \bar{n}^+ (r_s) \bar{u}^+ (r_s)}{\sqrt{c^2 - (\bar{u}^+ (r_s))^2}} \right) = 0,\\
			\frac{\d}{\d p_e} \left( \frac{\bar{p}^+ (r_s) + \bar{\rho}^+ (r_s) c^2}{\bar{n}^+(r_s) \sqrt{c^2 - (\bar{u}^+ (r_s))^2} } \right) = 0.
		\end{cases}
	\end{equation}
	The first law of the thermodynamics states
	\begin{equation}\label{law}
		T \d S = \d e + p \d (1/n).
	\end{equation}
	Then, a direct calculation yields
	\begin{equation}\label{derivation-1}
		\begin{cases}
			2 \bar{n}^+(r_s) \bar{u}^+ (r_s) \frac{\d r_s}{ \d p_e} + \frac{r_s \bar{n}^+(r_s) c^2}{ c^2 - (\bar{u}^+(r_s))^2} \frac{\d \bar{u}^+ (r_s)}{ \d p_e} + r_s \bar{u}^+(r_s) \frac{\bar{n}^+(r_s)}{ \d p_e} = 0,\\
			\bar{T}^+ (r_s) \frac{\d \bar{S}^+ (r_s)}{ \d p_e} + \frac{1}{\bar{n}^+ (r_s)} \frac{\d \bar{p}^+(r_s)}{ \d p_e} + \frac{\bar{u}^+ (r_s) (\bar{p}^+ (r_s) + \bar{\rho}^+ (r_s) c^2) }{\bar{n}^+(r_s) (c^2 - (\bar{u}^+ (r_s))^2)} \frac{\d \bar{u}^+ (r_s)}{ \d p_e} = 0.
		\end{cases}
	\end{equation}
	
	Since 
	\begin{equation*}
		\frac{\d}{\d p_e} \left(\frac{\bar{p} + \bar{\rho} c^2}{c^2 - \bar{u}^2} \bar{u}^2 \right) \\
		= 
		\frac{\bar{u} (\bar{p} + \bar{\rho} c^2)}{c^2 - \bar{u}^2} \frac{\d \bar{u}}{\d p_e} - \frac{\bar{p} + \bar{\rho} c^2}{c^2 - \bar{u}^2} \bar{u}^2 \frac{2 \d r_s}{r_s \d p_e},
	\end{equation*}
	then 
	\begin{equation}\label{derivative-2}\begin{aligned}
			& \frac{\d}{\d p_e} \left[ \frac{\bar{p} + \bar{\rho} c^2}{c^2 - \bar{u}^2} \bar{u}^2\right] (r_s) 
			= \frac{\bar{u}^+ (r_s) (\bar{p}^+ (r_s) + \bar{\rho}^+ (r_s) c^2)}{c^2 - (\bar{u}^+ (r_s))^2} \frac{\d \bar{u}^+ (r_s)}{\d p_e} \\
			& \quad
			- 
			\frac{\bar{u}^- (r_s) (\bar{p}^- (r_s) + \bar{\rho}^- (r_s) c^2)}{c^2 - (\bar{u}^- (r_s))^2} \frac{\d \bar{u}^- (r_s)}{\d p_e}
			- 
			\left[ \frac{\bar{p} + \bar{\rho} c^2}{c^2 - \bar{u}^2} \bar{u}^2\right](r_s) \frac{2 \d r_s}{r_s \d p_e}.
		\end{aligned}
	\end{equation}
	Moreover, using the second equation in \eqref{derivation-1}, we obtain
	\begin{equation}\label{derivation-p}\begin{aligned}
			& \frac{\d}{\d p_e} \left[ \bar{p} \right] (r_s) = - \frac{\bar{u}^+ (r_s) (\bar{p}^+ (r_s) + \bar{\rho}^+ (r_s) c^2)}{c^2 - (\bar{u}^+ (r_s))^2} \frac{\d \bar{u}^+ (r_s)}{\d p_e} \\
			& \quad
			- \bar{n}^+ (r_s) \bar{T}^+(r_s) \frac{\d \bar{S}^+ (r_s)}{ \d p_e} 
			+ \frac{\bar{u}^- (r_s) (\bar{p}^- (r_s) + \bar{\rho}^- (r_s) c^2)}{c^2 - (\bar{u}^- (r_s))^2} \frac{\d \bar{u}^- (r_s)}{\d p_e}.
	\end{aligned}\end{equation}
	Substituting \eqref{derivative-2} and \eqref{derivation-p} into $\left[ \frac{\bar{p} + \bar{\rho} c^2}{c^2 - \bar{u}^2} \, \bar{u}^2 + \bar{p} \right] (r_s) = 0$ yields
	\begin{equation}\label{derivation-3}
		\left[ \frac{\bar{p} + \bar{\rho} c^2}{c^2 - \bar{u}^2} \bar{u}^2\right] (r_s)  \frac{2 \d r_s}{r_s \d p_e} = - \bar{n}^+ (r_s) \bar{T}^+(r_s) \frac{\d \bar{S}^+ (r_s)}{ \d p_e} .
	\end{equation}
	
	Form \eqref{rel-const}, one also has
	\begin{equation}\label{rel-deriv}
		\begin{cases}
			\frac{\mathrm{d}}{\mathrm{d} p_e} \left( \frac{r_2^2 \bar{n}^+ (r_2) \bar{u}^+ (r_2)}{\sqrt{c^2 - (\bar{u}^+ (r_2))^2}} \right) = 0,\\
			\frac{\d}{\d p_e} \left( \frac{\bar{p}^+ (r_2) + \bar{\rho}^+ (r_2) c^2}{\bar{n}^+(r_2) \sqrt{c^2 - (\bar{u}^+ (r_2))^2} } \right) = 0.
		\end{cases}
	\end{equation}
	A straightforward computation gives 
	\begin{equation}\label{derivation-4}
		1 + c^2 \frac{\d \bar{\rho}^+ (r_2)}{ \d p_e} - \frac{\bar{p}^+ (r_2) + \bar{\rho}^+ (r_2) c^2}{\bar{n}^+ (r_2)} \cdot \frac{c^2 + (\bar{u}^+ (r_2))^2}{c^2} \frac{\d \bar{n}^+ (r_2)}{ \d p_e} = 0.
	\end{equation}
	It follows from \eqref{rho} and \eqref{eq:internal-energy} that 
	\begin{equation*}
		\frac{\d \bar{\rho}^+ (r_2)}{ \d p_e} = \frac{\d \bar{n}^+ (r_2)}{ \d p_e} + \frac{1}{(\gamma - 1)c^2}.
	\end{equation*}
	Combining this with \eqref{rel-sound} and \eqref{derivation-4} gives
	\begin{equation}\label{derivation-5}
		\frac{\gamma}{\gamma - 1} - \frac{c^2 (\bar{c}_s^2(r_2) + (\gamma - 1) (\bar{u}^+ (r_2))^2)}{(\gamma - 1)c^2 - \bar{c}_s^2(r_2)} \frac{\d \bar{n}^+ (r_2)}{\d p_e} = 0,
	\end{equation}
	where $\bar{c}_s^2(r_2) = c_s^2 (\bar{n}^+ (r_2), \bar{S}^+ (r_s))$.

	Substituting \eqref{eq:internal-energy} into \eqref{law}, we obtain
	\begin{equation}\label{law-2}
		T \mathrm{d}S = \frac{1}{\gamma - 1} \left( \frac{\mathrm{d}p}{n} - \frac{p}{n^2} \mathrm{d}n \right) - \frac{p}{n^2} \mathrm{d}n = \frac{1}{\gamma - 1} \frac{\mathrm{d}p}{n} - \frac{\gamma}{\gamma - 1} \frac{p}{n^2} \mathrm{d}n.
	\end{equation}
	Together with \eqref{p}, we have
	\begin{equation*}
		\frac{\d \bar{n}^+ (r_2)}{\d p_e} = \frac{1}{\gamma R \bar{T}^+ (r_2)} - \frac{(\gamma - 1)\bar{n}^+ (r_2)}{\gamma R} \frac{\d \bar{S}^+ (r_s)}{\d p_e}.
	\end{equation*}
	Substituting the above into \eqref{derivation-5} yields
	\begin{equation*}
		\frac{\d \bar{S}^+ (r_s)}{\d p_e} = \frac{(\bar{u}^+ (r_2))^2 - \bar{c}_s^2 (r_2)}{\bar{n}^+ (r_2) \bar{T}^+ (r_2) (\bar{c}_s^2 (r_2) + (\gamma - 1) (\bar{u}^+ (r_2))^2 )} < 0.
	\end{equation*}
	Hence, 
	\begin{equation*}
		\frac{\d r_s}{\d p_e} < 0 .
	\end{equation*}

	\textbf{Step 5.} Recall from \textbf{Step 1} that a unique smooth supersonic solution exists on $[r_1, r_s]$ for $r_s \in [r_1, r_2]$, while \textbf{Steps 2} and \textbf{3} yield the existence of a unique smooth subsonic solution on $[r_s, r_2]$. According to \textbf{Step 4}, the mapping $r_s \mapsto f(r_s) = \bar{p}^+ (r_2)$ defines a strictly monotonically decreasing continuous function on $[r_1, r_2]$. At $r_s = r_1$ and $r_s = r_2$, there exist two different constants $P_2$ and $P_1$, respectively, with $P_1 < P_2$. Consequently, the monotonicity established in \textbf{Step 4} guarantees that for any exit pressure $p_e \in (P_1, P_2)$, there exists a unique corresponding shock position $r_s \in (r_1, r_2)$.
\end{proof}

	{\bf Acknowledgment.} Weng is supported by National Natural Science Foundation of China (Grants No. 12571240, 12221001). 

{\bf Data Availability Statement.} No data, models or code were generated or used during the study.

{\bf Conflict of interest.} The author states that there is no conflict of interests.

\end{document}